\documentclass[12pt]{paper}
 \usepackage{geometry}
\usepackage{hyperref}
\usepackage{float}
                \usepackage{amsthm,amsfonts,amsmath,amscd,amssymb,epsfig,verbatim,mathabx
}

\usepackage{graphicx}
\usepackage{epstopdf}
 {\title{Making  cobordisms  symplectic} 
\author{Yakov Eliashberg\thanks{Partially supported by the NSF grants DMS-1505910 and DMS-1807270} \\ Stanford University  \and   Emmy Murphy\thanks{Partially supported by the NSF grant DMS-1906564}
\\Northwestern University} 
\date{}  
 \setlength{\marginparwidth}{1.2in}
\let\oldmarginpar\marginpar
\renewcommand\marginpar[1]{\-\oldmarginpar[\raggedleft\footnotesize #1]%
{\raggedright\footnotesize #1}}

\parindent=0pt
\parskip=4pt
\hyphenation{ma-ni-fold ma-ni-folds sub-ma-ni-fold sub-ma-ni-folds}
\theoremstyle{plain}
\newtheorem{theorem}{Theorem}[section]
\newtheorem{thm}[theorem]{Theorem}
\newtheorem{corollary}[theorem]{Corollary}
\newtheorem{cor}[theorem]{Corollary}

\newtheorem{prop}[theorem]{Proposition}
\newtheorem{lemma}[theorem]{Lemma}

\theoremstyle{remark}

\newtheorem{remark}[theorem]{Remark}
\newtheorem*{remark*}{Remark}

\newtheorem*{example*}{Example}

\theoremstyle{definition}
\newtheorem*{definition}{Definition} 
\newcommand{\wt}{\widetilde}
\newcommand{\wh}{\widehat}

\newcommand{\oLambda}{\overline{\Lambda}}

\newcommand{\p}{\partial}

\newcommand{\om}{\omega}

\newcommand{\eps}{\varepsilon}

\newcommand{\Z}{{\mathbb{Z}}}
\newcommand{\R}{{\mathbb{R}}}
\newcommand{\C}{{\mathbb{C}}}

\newcommand{\st}{{\rm st}}

\newcommand{\Int}{{\rm Int\,}} %Interior

\renewcommand{\max}{{\rm max}}

\newcommand{\std}{{\rm std}}

\newcommand{\Id}{\mathrm {Id}}

\newcommand{\Morse}{\mathrm{Morse}}
\newcommand{\Ker}{\mathrm{Ker\,}}

\newcommand{\ext}{\mathrm{ext}}
\newcommand{\intr}{\mathrm{int}}

\newcommand{\x}{\times}
\newcommand{\sm}{\setminus}

\def\Op{{\mathcal O}{\it p}\,}
\numberwithin{figure}{section}

\begin{document}
\maketitle
\begin{abstract}
We establish an  existence $h$-principle for symplectic cobordisms of dimension $2n>4$ with concave overtwisted contact boundary. \end{abstract}

%\tableofcontents
\section{Introduction}\label{sec:intro}

 \subsubsection*{Symplectic cobordisms}
 %%%%
  We say that $(W,\om,\xi_-,\xi_+)$  is a {\em  symplectic cobordism} between  contact manifolds $(\p_\pm W,\xi_\pm)$ if 
  \begin{description}
  \item{-}  $W$ is a smooth   cobordism between $\p_- W$ and $\p_+W$, and
 \item {-}  $\omega$ is a symplectic form which admits   a  Liouville vector field $Z$ near $\p W$  such that $Z$ is inwardly transverse to $\p_-W$,  outwardly transverse to $\p_+W$ and the contact forms $\lambda_\pm=\iota(Z)\omega|_{\p_\pm W}$ define the contact structures $\xi_\pm$. 
  \end{description}   
  All cobordisms we consider in this paper are assumed to be connected while their boundaries  $\p_\pm W$ could be disconnected.
  We recall that  a vector field $Z$ is called {\em Liouville} for  a symplectic form $\omega$ if    $d(\iota(Z)\omega)=\omega$.
  
  By a {\em Liouville cobordism } $(W,\lambda,\xi_-,\xi_+)$ between  contact manifolds $(\p_\pm W,\xi_\pm)$ we mean  a symplectic cobordism $(W,\om)$ between $(\p_\pm W,\xi_\pm)$   
  with a fixed primitive $\lambda$, $d\lambda=\omega$, so that the Liouville vector field $Z$ (defined by being the $\omega$-dual to $\lambda$) is inwardly transverse to $\p_-W$ and outwardly transverse to $\p_+W$.
   %%%%%%
 
An obvious  necessary condition for finding a symplectic cobordism structure between contact manifolds $(\p_-W, \xi_-)$ and $(\p_+W,\xi_+)$ on a fixed smooth cobordism $W$ is the  existence of an {\em almost symplectic} cobordism structure; that is     a non-degenerate  but not necessarily closed $2$-form $\eta$ on ${W}$ which coincides   with $d\lambda_{\pm}$ near $\p_\pm W$.

We call such $\eta$ an {\it almost} symplectic  cobordism structure. The homotopy class of an almost symplectic structure  is determined by the homotopy class   of an almost complex structure $J$, such that $J$ is tamed by $\eta$,   and 
$\xi_{\pm}$  are $J$-complex subbundles.

 Note that almost symplectic forms $\eta $ and $f\eta$, where  a $f$ is a positive function, are canonically homotopic. Hence,  given two almost symplectic form $\eta$ and $\wt\eta$   on a manifold $M$ which coincide on $TM|_A$  for a closed subset $A\subset M$, we say,  slightly abusing the terminology, that  $\eta$ and $\wt \eta$ are homotopic relative $A$  if they can be connected by a homotopy of almost symplectic forms $\eta_t$ such that $\eta_t|_{TM|_A}=f_t\eta|_{TM|_A}$ where $f_t:A\to\R$ is family of positive  functions.

The following theorem is the  main result of the paper:
\begin{thm}\label{thm:cobord}
Let $(W,\eta)$ be an almost symplectic   cobordism  of dimension $2n>2$ between non-empty contact manifolds  $(\p_\pm W,\xi_\pm)$. 
%Let  $\eta$ be an almost symplectic cobordism between $(\p_-W, \xi_-)$ and $(\p_+W,\xi_+)$. 
Suppose that
\begin{itemize}
 \item the contact structure $\xi_-$ is overtwisted on at least one of the
 connected components of $\p_-W$;
  \item if $n=2$  the contact structure $\xi_+$ is overtwisted  on at least one of the 
 connected components of $\p_+W$. 	 
  \end{itemize}
 Then there exists a Liouville cobordism structure $(W,\lambda,\xi_-,\xi_+)$  such that $d\lambda$  and $\eta$ are homotopic  as almost symplectic structures fixed on  $\p W$.
 \end{thm}
 As  Proposition \ref{4-concordance} below shows, the condition
that   the positive end of a $4$-dimensional cobordism has an overtwisted component is essential.
  On the other hand, Theorem \ref{thm:cobord} implies
 \begin{cor}\label{cor:4d}
 Let $W$ be a $4$-dimensional manifold with non-empty boundary $\p W$,  and let $\eta$ be  an almost symplectic form   such that  $\eta|_{\Op\p W}=d\lambda$  for a Liouville form $\lambda$ whose Liouville field is outwardly transverse  to $\p W$.  Then the (interior) connected sum $W':=W\#(S^3\times[0,1])$ admits a Liouville cobordism structure $\Lambda$ with $\p_+W':=\p W\amalg (S^3\times 1), \p_-W=S^3\times 0$ and $\Lambda|_{\Op\p W}=\lambda$.
 \end{cor}
 
  \begin{remark}\label{rem:at-least-one} 
 Note that  any two odd-dimensional manifolds which are  endowed with a stable almost complex structures are cobordant in an almost complex category, see \cite{Milnor, Nov60}. In particular, for any  contact manifolds  $(\p_\pm W,\xi_\pm)$ an almost symplectic cobordism $(W,\eta)$, as in Theorem \ref{thm:cobord}, always exists. 
  \end{remark}
Any stable almost complex structure on an odd-dimensional manifold is realized by an overtwisted contact structure, see \cite{BEM}. Thus the above theorem implies that any smooth cobordism $W$ with non-empty boundaries that admits an almost symplectic structure also admits a structure of a symplectic cobordism between two contact manifolds. 
 
 The notion of an {\em overtwisted} contact structure was introduced in
 \cite{Eli89} in dimension $3$  and   extended to all dimensions in \cite{BEM}. Without giving precise definitions, which  will be not important for the purposes of this paper, we summarize below  the main   results about overtwisted contact structures which will be used in this paper. Theorem \ref{thm:ot-properties}     is proven in  \cite{BEM} for   $n>1$ and in \cite{Eli89} for $n=1$. Theorem \ref{thm:ot-thick} is proven in  \cite{CaMuPr}. If $\xi=\{\alpha=0\}$ is an overtwisted contact structure, we will refer to $\alpha$ as an overtwisted contact form.

 All contact structures which appear in this paper will be coorientable. We recall that   a cooriented almost contact structure 
 is a  cooriented hyperplane field $\zeta \subset TM$ equipped with a  linear  symplectic structures $\eta$ on it, which is defined up to a positive conformal factor.  Cooriented almost contact structures  can be  represented by  pairs
$(\alpha,\eta)$ where $\alpha$ is a  non-vanishing 1-form on $M$ and $\eta$  is a non-degenerate two-form on the hyperplane field $\xi=\{\alpha=0\}$.  We note that given a maximal ($=\dim M-1$) rank   2-form $\eta$ on $M$ the corresponding 1-form  $\alpha$ can be reconstructed uniquely up to a contractible choice.   The existence of an almost contact structure is equivalent to the existence of a {\it stable almost complex structure} on $M$, i.e.\ a complex structure on the bundle $TM\oplus\eps^1$ where $\eps^1$ is the trivial line bundle over $M$.

 \begin{thm}\label{thm:ot-properties}  
 Let $\theta $ be a cooriented almost contact structure on a $(2n+1)$-dimensional manifold  $M$ which is  genuinely contact  on a neighborhood $\Op A$ of a closed subset $A\subset M$. Then there exists an overtwisted  contact structure $\xi$ on $M$  which coincides with $\theta$ on a neighborhood of $A$ and which is  homotopic   to $\eta$ through almost contact structures fixed on $A$.   Moreover, suppose we are  given  two  contact structures $\xi_0,\xi_1$ and a homotopy $\eta_t$, $t\in[0,1]$, of almost contact structures such that  
 \begin{itemize}
 \item[-] $\xi_0$, $\xi_1$ and $\eta_t$ coincide on $\Op A$ for all $t\in[0,1]$,
 \item[-] $\xi_0$ and $\xi_1$ are overtwisted on every connected component of $M\setminus A$. 
 \end{itemize} 
Then there exist
\begin{itemize}
\item[-] a homotopy $\xi_t$, $t\in[0,1]$, of genuine contact structures connecting $\xi_0$ and $\xi_1$, and 
\item[-] a $2$-parametric family $\Theta_{t,s}$ of almost contact structures
\end{itemize} such that 
\begin{itemize}
\item [-] $\Theta_{t,0}=\xi_t,\; \Theta_{t,1}=\theta_t,\;
\Theta_{0,s}=\xi_0\;\hbox{and}\; \Theta_{1,s}=\xi_1$ for all $t,s\in[0,1]$.
\item[-] $\Theta_{t,s}|_{\Op A}=\xi_0|_{\Op A}.$
\end{itemize}
 \end{thm}
 \begin{thm}\label{thm:ot-thick} 
   Let  $(M,\{\alpha=0\})$  be  an overtwisted contact manifold  and $(D_R,\mu)$ the disc of radius $R$ in $\R^2$ endowed with a Liouville form  $\mu =xdy-ydx$. Then for  sufficiently large values of  $R$, the product $(M\times D_R,\{\alpha\oplus\mu=0\})$ is also overtwisted.
  \end{thm}
%%%Added
%%%%%%%%%%%%%%%%%%%%%%%%%%%%%

 We note that any contact manifold of dimension $>1$ can be made overtwisted without changing its almost contact homotopy class by a modification in a neighborhood of one of its points, and that any  two definitions of overtwistedness for which Theorem \ref{thm:ot-properties}  holds are equivalent.  
 
 For any closed form $\om$ on $W$  equal to 
$d\lambda_\pm$  on $\p_\pm  W$  one  can canonically associate a cohomology class $[\om]\in H^2(W,\p W)$.  Indeed, take any $1$-form $\lambda$ on $W$ extending $\lambda_\pm$  Then    the cohomology class of $[\om-d\lambda]\in H^2(W,\p W)$ is independent of the choice of the extension $\lambda$.

 \begin{cor}\label{cor:cohomology-class}
 Under the assumptions of Theorem \ref{thm:cobord} there is a symplectic cobordism structure $\omega$ on $W$ between $\xi_-$ and $\xi_+$ so that $[\omega]$ is equal to any given cohomology class $a\in H^2(W,\p W)$.
 \end{cor}
 Indeed, let $\lambda$ be the Liouville form provided by Theorem \ref{thm:cobord}, and let $\sigma$ be  any closed form with compact  support  representing $a\in H^2(W,\p W)$. Then for a sufficiently large constant $C$  the form $\omega = Cd\lambda+\sigma$ is symplectic and has the required properties.
 \medskip
 
It was proven  in \cite{CaMuPr}, see also   Corollary \ref{cor:concordance} below, that if $n>2$ then for {\em any} contact structure $\xi$ on $\p_+W$ there is a Liouville concordance (i.e. a Liouville cobordism diffeomorphic to $\p_+W\times[0,1]$) between $\xi$ on the positive end and an overtwisted contact structure $\xi_{\rm ot}$ on the negative end. This is not the case for $n=2$, as Proposition \ref{4-concordance} below  demonstrates. Hence, if $n>2$ it is sufficient to prove Theorem \ref{thm:cobord} for the case when both contact structures $\xi_\pm$ are overtwisted. This is the only reason why we need an additional hypothesis when $n=2$.

 \subsubsection*{Symplectic manifolds with a conical singularity}
  Given a   $2n$-dimensional manifold $X$ with boundary $\p X$, a symplectic form  $\omega$ on    $X\setminus p$, $p\in X$,  and a contact structure $\xi$ on $\p X$, we say that $(X,\omega,\xi)$ is  
 a {\em symplectic domain   with a conical singularity at $p$ and contact boundary $(\p X,\xi)$} if  $(\p X,\xi)$ is a positive contact boundary in the above sense, and near  
 $p$ there exists a Liouville field $Z$  and a Morse function $\phi$ with the minimum at $p$   such that $d\phi(Z)>0$.   In other words,  in  a punctured  ball centered at $p$ the form $\omega$   is symplectomorphic to the negative part of the symplectization of a contact structure $\zeta$ on the boundary sphere
$S^{2n-1}$. We will call $(S^{2n-1},\zeta)$ the \emph{link} of the singularity $p$.
By removing from $X$ a ball $B$ centered at $p$, we can equivalently view a symplectic structure $\om$ on $X$ with a conical singularity at $p$   as a symplectic cobordism structure on 
$\dot{X}=X\setminus\Int B$  between   $(\p_-\dot X:=\p B,\zeta)$  and $(\p_+\dot X:=\p X,\xi)$.

Of course, if the contact structure  $\zeta$ is standard, then the form extends to a non-singular symplectic form on the whole  $X$. Note that if $\zeta$ is overtwisted then according to
Theorem \ref{thm:ot-properties} $\zeta$ is uniquely determined up to isotopy by the homotopy class of  $\om$  in the space of non-degenerate $2$-forms on  a punctured neighborhood of $p$.  

The following result is a special case of Theorem \ref{thm:cobord}.
\begin{thm}\label{thm:conical-sing}
Let $X$ be a  compact   $2n$-dimensional manifold  with  non-empty boundary $\p X$,  and let $\xi$ a   contact structure on $\p X$.  Let $\eta$ be a non-degenerate $2$-form on $X\setminus p$, $p\in X$, which  is equal to $d\lambda$ near $\p X$,  where $\lambda$ is any 1-form satisfying $\xi=\{\lambda|_{\p X}=0\}$ and such that the Liouville  vector field dual to $\lambda$ is outward transverse to $\p X$.
  Let 
 $a\in H^2(X,\p X)$  be  a relative  cohomology class. If $n=2$ assume, in addition, that $\xi$ is overtwisted. Then there exists  a    symplectic structure $\om$ on $X$ with a conical singularity at $p$ with an overtwisted link $(S^{2n-1},\zeta)$ and positive contact boundary $(\p X,\xi)$,   such that   the cohomology class $[\om]\in H^2(X\setminus p,\p X)=H^2(X,\p X)$ coincides with $a$,  and the homotopy class of $\om|_{X\setminus p}$ as a non-degenerate form   coincides with the  rel. $\p X$ homotopy class of $\eta$.  
\end{thm}

\begin{remark}{\rm 1.
In contrast with Theorem  \ref{thm:conical-sing}, the  construction of {\em non-singular symplectic structures} on $X$ is severely constrained.  For instance, according to  Gromov's theorem  in dimension $4$ and Eliashberg-Floer-McDuff's theorem in higher dimensions, see \cite{Gro85, McD91}, any  symplectic manifold $(X,\om)$ bounded by the standard contact sphere and  satisfying the condition
$[\om]|_{\pi_2(X)}=0$   has to be  diffeomorphic to a ball. 

\noindent 2.  It is interesting to compare the flexibility phenomenon for symplectic structures with   conical singularities with  overtwisted links,   exhibited in Theorem \ref{thm:conical-sing},   with a similar  flexibility phenomenon for  Lagrangian manifolds with conical singularities with loose Legendrian links, see   \cite{EliMur13}.
}
\end{remark}
Theorem \ref{thm:conical-sing} implies the following
\begin{cor}\label{cor:conn-sum}
Let $X$ be  a closed manifold of dimension $2n>4$ that admits an almost complex  structure on $X\setminus p$, $p\in X$. Let $a\in H^2(X)=H^2(X\setminus p)$ be any cohomology class. Then for any closed symplectic $2n$-dimensional  manifold $(Z,\omega)$,
the connected sum $X\# Z$ admits a symplectic form with a conical singularity in the cohomology class $a+C[\omega]$ for a sufficiently large constant $C>0$.
\end{cor}
\begin{proof}
Let $B$ be a ball centered at $p$. The obstruction to extending an almost complex structure $J$ from $X\setminus \Int B$ to $X$  is an element  $\alpha\in\pi_{2n-1}(SO(2n)/U(n))$. 
Choose two disjoint  balls $B_1, B_2\subset \Int B$ and define $J$ on a neighborhood $\Op\p B_1$ of $\p B_1$ as the push-forward of the standard complex structure on the boundary of a ball    $D\subset \C^n$ under an orientation preserving diffeomorphism of $h:\Op\p D\to\Op\p B_1$ such that
$h(\p D)=\ B_1$, and  $h$ sends the outward normal vector field to  the boundary $  D$ to the inward normal vector field  to the boundary of $  B_1$. Note that  $J$   does not extend to an almost complex structure on $B_1$, unless $n=1,3$. Let $\beta\in\pi_{2n-1}(SO(2n)/U(n))$ be the obstruction to this extension.
We furthermore   extend $J$ to $B\setminus \Int(B_1\cup B_2)$ and note  that the obstruction to extending  $J$ to $B_2$ is equal to $\alpha-\beta$.
 By construction the almost complex structure $J$ on $\Op\p B_1$ is  compatible  with the
 push-forward  $h_*\xi_{\std}$ of  the  standard contact structure $\xi_\std$ on $\p D$, co-oriented by the outward normal.  Next, we apply Theorem \ref{thm:conical-sing} to get a symplectic  form  $\omega_X$ on $X \setminus \Int B_1$   with a conical singularity at the center of $B_2$, so that $[\omega_X] = a$ and $\p_+(X \setminus \Int B_1, \omega_X) \cong (S^{2n-1}, \xi_\std)$.  Then $(X \setminus \Int B_1, \omega_X)$ can be implanted into the symplectic manifold $Z$,   after rescaling the symplectic form of $Z$ if necessary.
\end{proof}

As  we already pointed out above, when $n=2$, the assumption that  at 
least one of  components of the positive boundary is overtwisted is essential, as the following proposition proven in   \cite{MR, KM} demonstrates:
 
\begin{prop}\label{SW}\label{4-concordance}
Suppose  $(W,  \omega,\xi_-,\xi_+)$ is  an exact  
$4$-dimensional symplectic cobordism   and the contact structure $\xi_-$ is  symplectically fillable. Then there is no exact symplectic cobordism structure $(W,\wt\omega,\wt\xi_-, \xi_+)$   where $\omega$ and $\wt\omega$ are in the same homotopy class of almost symplectic forms  and $\wt\xi_-$ is overtwisted.
In particular, for any fillable contact manifold $(M, \xi)$ there is no symplectic concordance $(M\times[-1,1], \om)$ in either direction between  $(M, \xi)$ and $(M,\wt\xi)$ with   overtwisted $\wt\xi$.
\end{prop}
Note that the non-existence of a symplectic cobordism
between a fillable contact structure on the negative end, and an overtwisted structure on the positive one is a universal fact which holds for all symplectic cobordisms in all dimensions, see \cite{Gro85, Eli91,Nieder, BEM}.

 We do not know whether the condition $\p_+W\neq\varnothing$ in Theorem \ref{thm:cobord} is essential when $n>2$. We note, however, that Theorem \ref{thm:cobord} implies that {\em every overtwisted contact manifold  $(M,\xi)$ of dimension $>3$ admits a symplectic cap, i.e. there exists a symplectic cobordism $(W,\om)$ with $\p_+W= \varnothing$ and    $\p_-(W,\om)=(M,\xi)$.}
 Indeed, the group  of complex bordisms is trivial in odd dimensions,
 see \cite{Milnor,Nov60}. Hence, according to Theorem \ref{thm:cobord} there is a symplectic cobordism between $(M,\xi)$ on the negative end and the standard contact sphere on the positive one, which then can be capped by any closed symplectic manifold, as in the proof of Corollary \ref{cor:conn-sum}. An existence of a symplectic cap for all (and in particular) overtwisted contact $3$-manifolds  was originally proven in \cite{EtHo02}, see also  \cite{Eli03} and \cite{Et04}.  However, these papers do not  give much on the  topology of the symplectic cap.

\subsubsection*{Conformal symplectic manifolds}  
A {\em   conformal symplectic structure} on a manifold $M$ is given by an atlas of symplectic charts
$(U_i, \om_i)$, such that the transition maps $f_{ij}$ are {\em conformally} symplectic, i.e. $f^*_{ij}\om_{i}=c_{ij}\om_j$  for positive constants $c_{ij}\in\R$.

Equivalently, a conformal symplectic structure can be defined   as a symplectic structure  with coefficients in a flat principal bundle with fiber the  multiplicative group $\R_+$ of positive real numbers; such a bundle can be described by   a representation $\theta:\pi_1(M)  \to \R_+$.   In other words, a conformal symplectic structure  on $M$ is a symplectic structure  $\omega$ on   the universal cover $\wh M$  of $M$ such that  the action of $\pi_1(M)$ by deck transformations on $\wh M$ satisfies $g^ *\om=\theta(g)\om$ for any $g\in\pi_1(M)$. 
The representation $\theta:\pi_1(M)\to\R_+$ factors through a  homomorphism $\overline\theta: H_1(M,\R)\to \R_+$, so that  $\mu:=\log\overline\theta: H_1(M,\R)\to \R$ is an additive homomorphism, which   therefore   defines a cohomology class
  in $H^1(M,\R)$   which we will denote  by $ \mu_\om$.
 Theorem \ref{thm:cobord} implies the following $h$-principle for conformal symplectic structures.
 \begin{theorem}\label{thm:conformal}
 Let $(M, \eta)$ be a closed $2n$-dimensional  almost symplectic manifold and  suppose $\mu\in   H^1(M,\Z)$ is a non-zero cohomology class. Then there exists a conformal symplectic structure $\om$  in the formal homotopy class of  $\eta$,   with $\mu_\omega=c\mu$ for some real $c\neq 0$.  \end{theorem}
  
\begin{proof}
 
There exists a smooth map $f:M\to S^1$ so that $\mu=f^*a$, where   $a$ is the generator of $H^1(S^1,\Z)$. Denote $\Sigma:=f^{-1}(p)$ for a regular value $p$ of $f$.
By cutting $M$ open along $\Sigma$ we get a cobordism $W$ with $\p_\pm W\cong\Sigma$.
Let us endow the  boundary components $\p_\pm W$ with copies of the unique overtwisted contact structure $\xi$ determined by the almost complex structure $J$.
  Theorem \ref{thm:cobord} yields a symplectic   cobordism structure $\omega$ on $W$ with the prescribed contact boundaries, and in  the relative homotopy class determined by $\eta$.  
We can assume that the contact structures on $\p_\pm W\cong\Sigma$  are  given by   contact
 forms $\alpha_{\pm}$, such that $\alpha_+=k\alpha_-$ for some constant  $k>0$.  
Hence by identifying the symplectic forms $k\om|_{\Op\p_-W}$ and $\om|_{\Op\p_+W}$ we get a conformal symplectic structure $\overline\om$ on $M$ corresponding to  the cohomology class $\mu_{\overline\om}=(\log k) \mu$. 
\end{proof}
\begin{remark} M. Bertelson and G. Meigniez informed us that  by further developing the  methods of this paper they proved in their forthcoming paper a stronger  version of Theorem \ref{thm:conformal}, establishing  existence of a conformal symplectic structure with any prescribed non-zero cohomology class $[\mu]\in H^1(M,\R)$.
\end{remark}
 \medskip 

\subsubsection*{Historical remarks}

The first constructions of symplectic cobordisms  between contact manifolds were based on the Weinstein handlebody construction, see \cite{Eli90,Wei91,CieEli12} and Theorem \ref{thm:Weinstein-case} below.  Examples of Liouville domains  with disconnected contact boundary (and hence non-Weinstein) were  constructed in \cite{McDuff} in dimension $4$, see also \cite{Mitsu}. This yields construction of Liouville domains of dimension $2n$ with     non-trivial homology up to dimension 
 $\sim
 \frac{3n}2$, see \cite{EliGro91}. 
   High-dimensional Liouville domains with disconnected boundary were constructed in  \cite{Geiges} and \cite{MaNiWe13}.   See  \cite{Gay07,We13, MaNiWe13} for more  constructions of  non-Weinstein symplectic  cobordisms. Related   problems  concerning  (different flavors of) symplectic fillability and the  topology of symplectic fillings were extensively studied, especially in the contact $3$-dimensional case.  See Chris Wendl's blog   \cite{CW-blog} 
 for a survey and a  discussion of   related  results and problems. 
  
  \subsubsection*{Acknoledgements}
  We thank Chris Wendl for providing some of the above references  and   Nikolai Mishachev for making  several pictures.
  We are also grateful to  the anonymous referees for constructive critical remarks and useful suggestions, and to Oleg Lazarev  and Dylan Cant for the attentive reading of the manuscript.  The authors are grateful to the following institutions for their support and hospitality: RIMS, Kyoto, and ITS ETH, Zurich (the first author), and IAS, Princeton (the second author).    

%%%%%%%%%%%%%%%%%%%%%%%%%%%%%%
\section{Topological and symplectic preliminaries}
\subsection{Sutured cobordisms}\label{sec:sutured}
\begin{definition}\label{def:sutured}
A {\em sutured} cobordism is a  compact manifold $W$ with boundary with corners together with a vector field $Z$ on $\Op\p W$ such that
\begin{itemize}
\item the boundary $\p W$ is presented as the union of two manifolds $\p_- W$ and $\p_+W$ with common boundary  $\p^2W=\p_+W\cap \p_-W$, along which  $\p W$  has a corner\footnote{i.e. each point of $\p^2W$ has a neighborhood diffeomorphic to $\R^{n-2}\times\{(x,y)\in\R^2;\, x,y\geq 0\}$, $n=\dim W$.};
\item the vector field $Z$ is inwardly transverse to $\p_-W$ and outwardly transverse to $\p_+W$.
\end{itemize}
\end{definition}

\begin{figure}[h]
\begin{center}
\includegraphics[scale=.6]{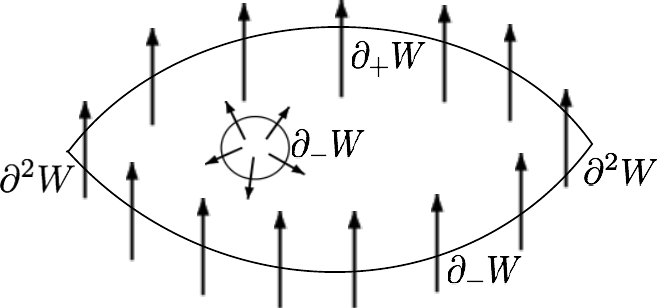}
\caption{A  sutured cobordism $W$.  Note that $\p_\pm W$ may be disconnected. In the figure, $\p_-W$ has one closed component and one component with boundary.}
\label{fig:corncob}
\end{center}  
\end{figure}

A traditional cobordism between closed manifolds $\p_-W$ and $\p_+W$ is a special case of a sutured cobordism when $\p^2W=\varnothing$. In the rest of the paper all considered  cobordisms are allowed to be sutured unless it is stated otherwise.
\begin{figure}[h]
\begin{center}
\includegraphics[scale=.6]{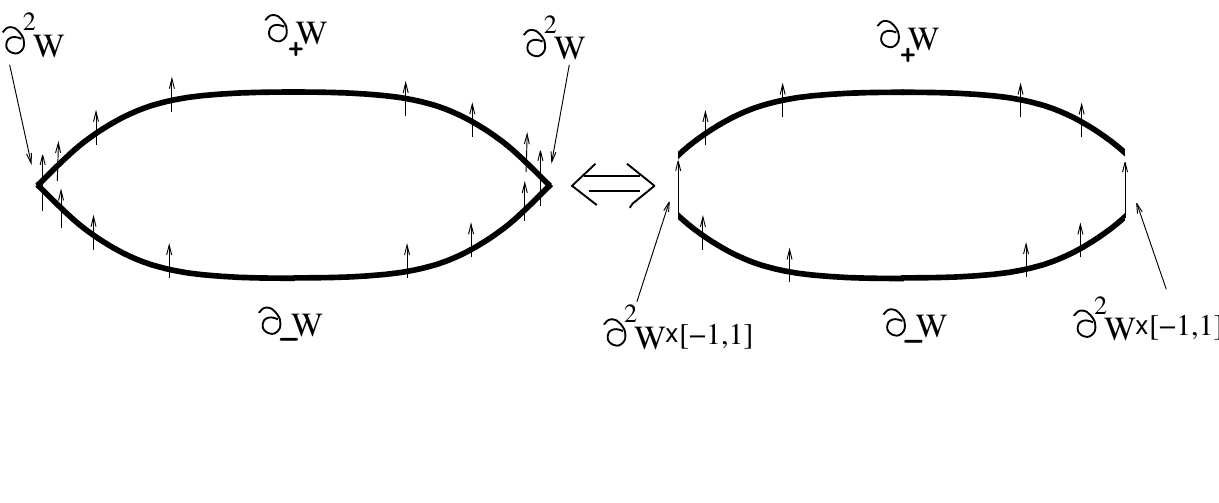}
\caption{Shaving off the corner of a sutured cobordism}
\label{fig:shaving}
\end{center}  
 \end{figure}

 \begin{figure}[h]
\begin{center}
 \includegraphics[scale=.6]{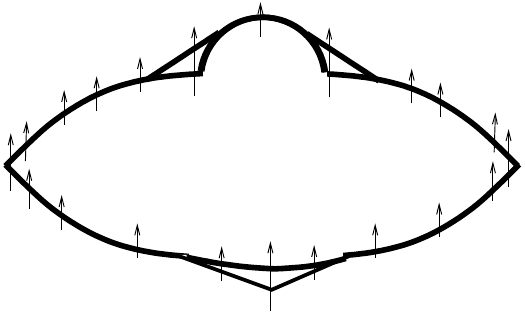} 
\caption{Smoothing non-essential corners of a sutured cobordism.}
\label{fig:smoothing}
\end{center}  
\end{figure}

 \begin{remark}\label{rem:cob} a) The sutured cobordism notion is parallel to the notion of a cobordism between manifolds with boundary, when the cobordism is required to be trivial over the boundary. Such a cobordism can be obtained from a sutured cobordism  by shaving off a neighborhood of the corner $\p^2W$ formed by trajectories of  the vector  field $Z$ originating from an open collar of $\p(\p_-W)=\p^2W$ in $\p_-W$, see Figure \ref{fig:shaving}. 
  
  To go  back to a sutured cobordism, one collapses the part $ \p(\p_-W)\times[0,1]$ of $\p W$ into $\p^2W$.
  In particular given a manifold $N$ with boundary we can  associate with the trivial cobordism $N\times[-1,1]$ its   sutured version $\wh N$ as follows.
  Choose a smooth function $\psi:N\to[0,1]$ such that $\p N$ equals the regular   level set $\{\psi=0\}$. Define
 $\wh N :=\{(x,t)\in N\times[-1,1]; -\psi(x)\leq t<\psi(x).\}$. We refer to $\wh N$ as the {\em sutured version} of the trivial cobordism $N\times[-1,1]$.

 \begin{figure}[h]
\begin{center}
\includegraphics[scale=.6]{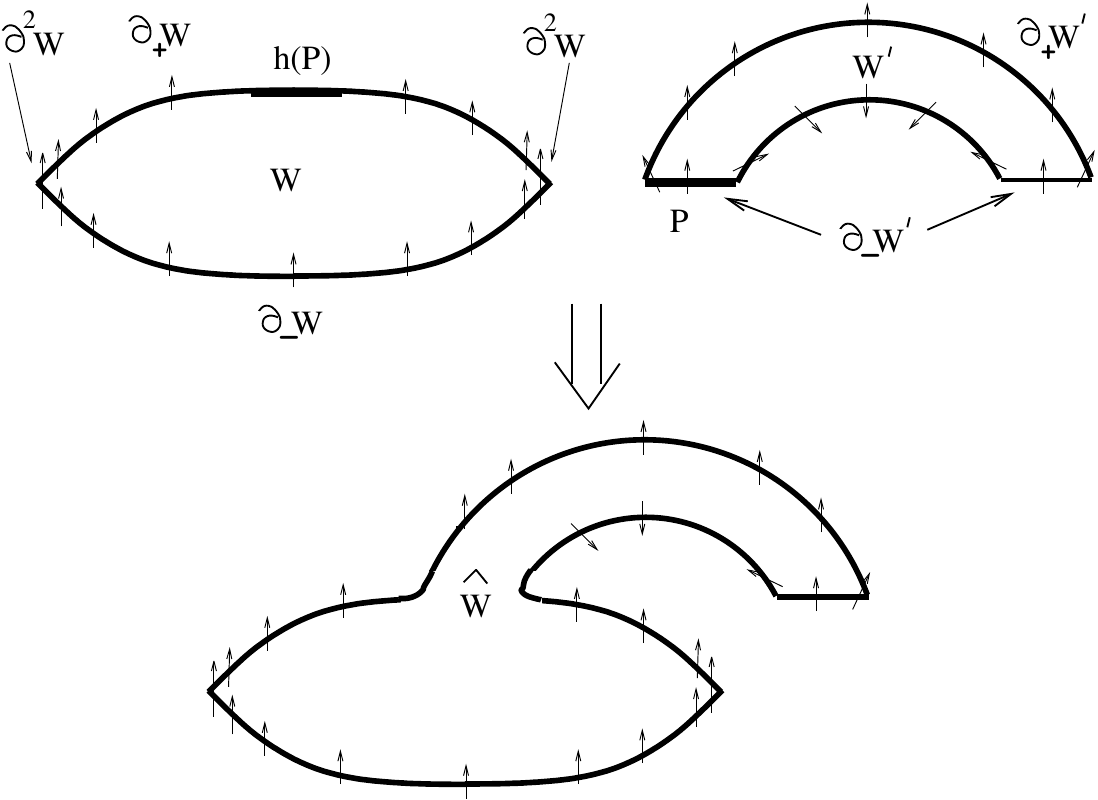}
\caption{concatenation of sutured cobordisms}
\label{fig:stacking-direct}
\end{center}  
\end{figure}
\begin{figure}
\begin{center}
\includegraphics[scale=.6]{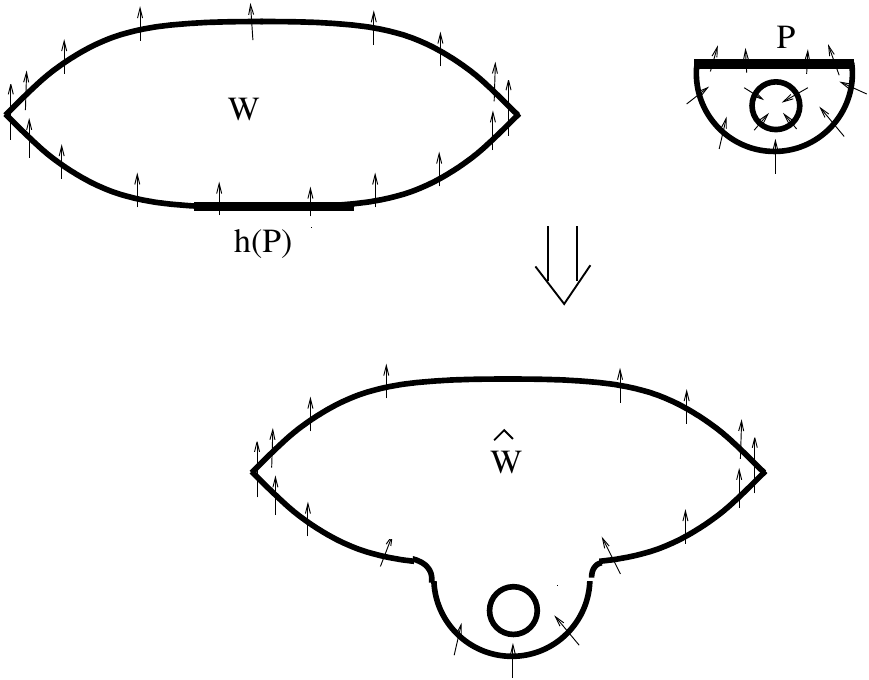}
\caption{Backward concatenation}
\label{fig:stacking-inverse}
\end{center}  
 \end{figure}

b)  The sutured cobordisms which   appear  in our constructions may sometimes have additional {\em non-essential corners} characterized  by  the property that the corresponding vector field $Z$ is transverse in the same sense (i.e. inwards or outwards) to all  boundary faces adjacent to the corner. Such a corner  can  always be removed  by  smoothing   the boundary $\p W$ in  a neighborhood of  the corner,   while keeping the boundary faces transverse to $Z$. See Fig. \ref{fig:smoothing}.

\end{remark}

 Let us discuss some operations one can perform on sutured cobordisms.  We do not specify below  the choice of the  required vector field  $Z$ when it is clear from the context or irrelevant. 

{\sl Concatenation.} Let $W$ and $W'$ be two sutured cobordisms 
and let $P$  be a connected  component of $\p_-W'$. Given an embedding
$h:P\to\Int\p_+W$ one can glue $W'$ to $W$ along $P$,  forming a manifold with boundary with corners  $\wh W:=W\mathop{\cup}\limits_h W'$. The smooth structure is determined by 
 matching the  vector fields $Z$ and $Z'$.
The resulting manifold $\wh W$ can be viewed as a sutured cobordism $$\left(\wh W,\;\p_-\wh W:=\p_-W\cup(\p_-W'\setminus P),\;\p_+\wh W:=(\p_+W\setminus \Int P)\cup\p_+W'\right)$$ (after smoothing all non-essential corners along $\p P$).
We say that $\wh W$ is a result of {\em concatenation} of $W$ and $W'$, see Figure \ref{fig:stacking-direct}.
 We will sometimes consider also an operation of  {\em backward concatenation}, see Figure~\ref{fig:stacking-inverse}. For that  we choose a connected   component  $P$ of $\p_+W'$ and an embedding
$h:P\to\Int\p_-W$,  then we glue $W'$ to $W$ along $P$, forming   $\wh W:=W'\mathop{\cup}\limits_h W$,
while also matching the   vector fields $Z'$ and $Z$.
The resulting manifold $\wh W$ can be viewed as a sutured cobordism 
$$(\wh W,\p_+\wh W:=(\p_+W'\setminus P)\cup\p_+W ,\p_-\wh W:=(\p_-W\setminus \Int P)\cup\p_-W'),$$  after smoothing all non-essential corner along $\p P$.

A special case of the concatenation operation is  handle attachment.  In that case $W'$ is the handle $D^k\times D^{n-k}$ with $\p_-W'=\p D^k\times D^{n-k}$, $\p_+W'=D^k\times\p D^{n-k}$ and the  and $P=\p_-W'$. To attach a handle  to another sutured cobordism W, we require
an embedding  $P=\p D^k\times D^{n-k}\to \Int\p_+W$, and then  apply the concatenation operation described above.
\footnote{Throughout the paper we use the letters $D$ and $B$ to denote closed and open  discs, respectively.
}

 The next   operation we will consider is another important instance of the concatenation operation.

\begin{figure}[h]
\begin{center}
 \includegraphics[scale=.55]{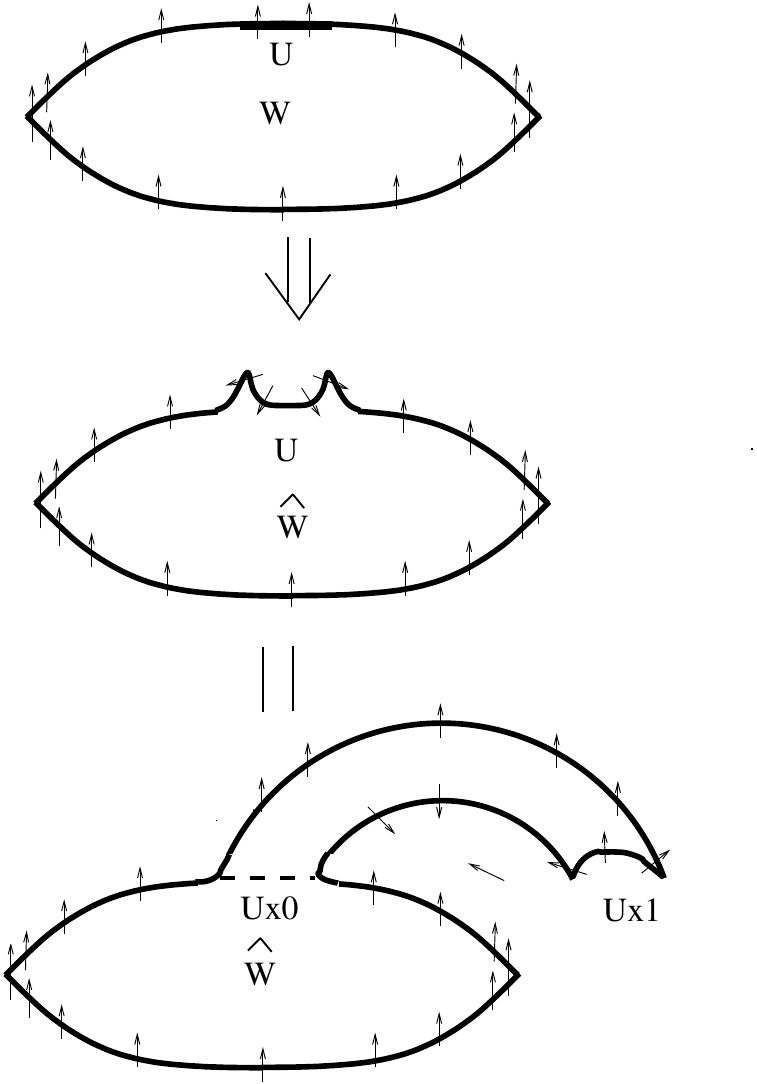} 
\caption{Inverting the domain $U \subset \p_+W$}
\label{fig:inversion}
\end{center}  
\end{figure}

{\sl Boundary inversion.}
 Let $(W; \p_-W,\p_+W)$  be   a   sutured  cobordism,  and let $U\subset \Int\p_+W$  be a  compact domain with smooth boundary. 
   Let us recall that the definition of a sutured cobordism $W$ involves a vector field $Z$ on $\Op\p W$ which is inwardly transverse to $\p_-W$ and outwardly transverse to  $\p_+W$. Let us modify $\p_+W$ by creating a corner along $\p U$ and changing $Z$     on $\Op U$ in such a way that   the new vector field $\wh Z$  is  inward transverse to $U$ and outward transverse to $\p_+W\setminus\Int U$. 
   We say that   the resulting  cobordism $$(\wt W=W, \p_-\wt W:=\p_-W\cup U,\p_+\wt W=\p_+W\setminus\Int U; \wt Z)$$   is obtained from
$(W; \p_-W,\p_+W)$ by {\em inverting}  the domain $U\subset\p_+  W$.  

The inversion operation can be equivalently described as concatenation as follows.
  Take $W':=U\times[0,1]$ and view it as a sutured cobordism $(W', \p_-W=(U\times 0\cup U\times 1),\p_+W:=\p U\times[0.1]).$ Let $\wh W$ be the result of concatenating $W$ and $W'$ by  identifying the component  $U\times 0\subset\p_-W'$ with $U\subset\Int\p_+W$. The cobordism $\wh W$, after smoothing non-essential corners,  is diffeomorphic to $\wt W$, see Figure \ref{fig:inversion}.

 {\sl Suture gluing.}
 Let $W, W'$ be sutured  cobordisms with sutures $S=\p^2W, S'=\p^2W'$. Consider domains $P\subset S, P'\subset S'$ with smooth boundaries. Given a diffeomorphism $g:P\to P'$ one can glue the cobordisms $W$ and $W'$ to get a sutured cobordism $\wh W$ with $\p_+\wh W=\p_+W\mathop{\cup}\limits_g\p_+W', \p_-\wh W=\p_-W\mathop{\cup}\limits_g\p_-W'$. The new suture $\wh S$ is equal to $\overline{S\setminus P}\mathop{\cup}\limits_g\overline{S'\setminus P'}$ which is obtained by shaving parts $P,P'$ of the sutures and gluing   the resulted boundary  components $P\times I$ and $P'\times I$. See Figure  \ref{fig:suture-gluing}.
  
\begin{figure}[h]
\begin{center}
 \includegraphics[scale=.65]{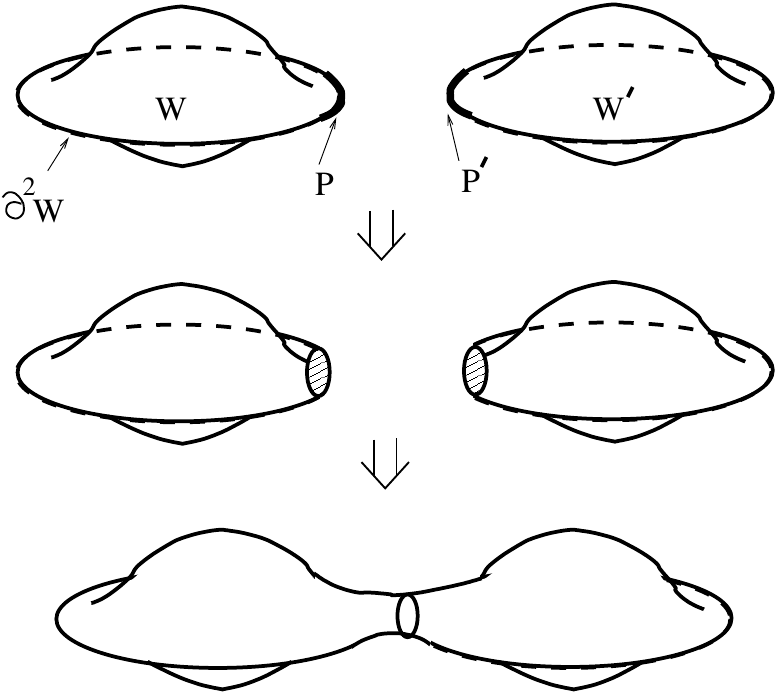} 
\caption{Suture gluing.}
\label{fig:suture-gluing}
\end{center}  
\end{figure}
 \medskip
 
A Morse  function $\phi:W\to\R$ on a sutured cobordism $(W,Z)$ is called {\em admissible} if  the vector  field  $Z$ is  gradient-like for  $\phi|_{\Op\p W}$. 
We stress the point  that  even when $\p^2W=\varnothing $, we {\em do not require $\phi$ to be constant on $\p_\pm W$} (in contrast with the usual set-up of Morse theory). If  $\p^2W=\varnothing $    one can always modify $\phi$ on $\Op \p W$  without creating additional critical points to make it constant on the boundary components.
\begin{definition}\label{def:Morse-type} Given a (sutured) cobordism $(W,\p_-W,\p_+W)$ we define    the {\em relative Morse type} $\Morse(W,\p_-W)$ of $(W,\p_-W)$
 as the minimal integer $m$ such that
$W$ admits an admissible Morse function $\phi:W\to\R$ whose critical points have index $\leq m.$ If $\p_-W=\varnothing$ we will write $\Morse(W)$ instead of $\Morse(W,\varnothing)$.  \end{definition}

 Note that   the pair  $(W,\p_+W)$ is $k$-connected, where $k=\dim W-\Morse(W,\p_-W)-1$.

\begin{lemma}\label{lm:Morse estimates}
\begin{enumerate}
\item Let $\wh W$ be the result of concatenation  or suture gluing of sutured cobordisms $W$ and $W'$.
Then $$\Morse(\wh W,\p_-\wh W)\leq\max(\Morse(W,\p_-W),\Morse(W',\p_-W')).$$
\item  Suppose that    the cobordism $(\wh W,\p_-\wh W,\p_+\wh W)$ is obtained from
$(W; \p_-W,\p_+W)$ by   inverting   the domain $U\subset\p_+  W$. Then
$$\Morse(\wh W,\p_-\wh W)\leq \max(\Morse(W,\p_-W),\Morse(U)+1).$$
\end{enumerate}
\end{lemma}
\begin{proof}
The proof of statement (i) is straightforward. Statement (ii) follows from (i) and the description of the boundary inversion as a concatenation of $W$ and $U\times I$, taking into account that $\Morse(U\times I,U\times\p I)\leq \Morse(U )+1.$
\end{proof}

\begin{figure}[h]
\begin{center}
 \includegraphics[scale=.55]{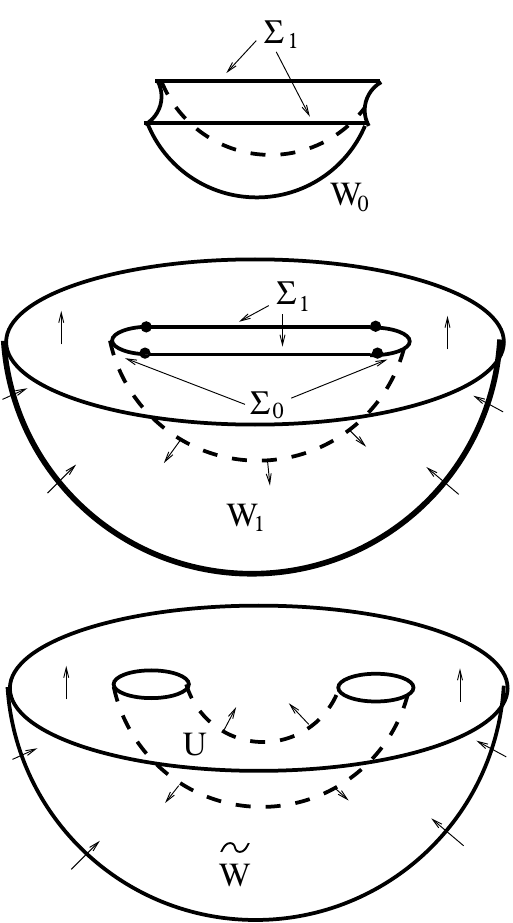} 
\caption{Building  cobordism $\wt W$}
\label{fig:drilling1}
\end{center}  
\end{figure}
\begin{lemma}\label{lm:remove}
Let $W$ be an $m$-dimensional sutured cobordism  and $N$ a $k$-dimensional, $k<m-1$, manifold with boundary. Let $\wh N$ be the sutured version of the trivial cobordism $N\times[-1,1]$, see Remark \ref{rem:cob}a). Suppose we are given an embedding $\phi:(\wh N,\p_+\wh N=N)\to(W,\Int \p_+W)$ which is transverse to $\p_+W$.
Let $U$ be a tubular neighborhood of $\p_- \wh N$ in $W$.  Consider $\wt W:=
\overline{W\setminus U}$, viewed as   a sutured cobordism with $\p_-\wt W=\p_-W\cup( \Sigma:=\overline{\p U\setminus\p_+W})$ and $\p_+\wt W:=\overline{\p_+W\setminus U}.$ Then
$$\Morse(\wt W,\p_-\wt W)\leq\max(\Morse(W,\p_-W),\Morse(N )+1).$$
  \end{lemma}
  \begin{proof}
     The cobordism $(\wt W,\p_-\wt W,\p_+\wt W)$ can be equivalently described as follows.  Choose any Riemannian metric on $\p_+W$ and  for a sufficiently small $\eps>0$ consider an $\eps$-neighborhood $V$ of $\phi(N)=\phi(\p_+\wh N)$ in $\p_+W$. The  neighborhood $V$ is the union  of 
     \begin{itemize}
     \item[(i)] the total space  of  a bundle  over $\phi(\p N)$ whose fibers are  $(m-k+1)$-dimensional half-discs of radius $\eps$    and 
     \item[(ii)] the  total space of the normal $\eps$-disc bundle of rank $m-k$ over $\phi(N)$.  
     \end{itemize}
      We denote the total spaces of the associated  half-sphere and  sphere bundles by $\Sigma_0$ and $\Sigma_1$, respectively.
    We have $\p V=\Sigma_0\cup\Sigma_1$.  
     Consider the  sutured   cobordism $W_1$   obtained from $W$ by inverting the domain $V$ in $\p_+W$, and let  $W_0$ be  the sutured version  of the trivial cobordism $V\times[-1,1]$.
      Note that $\p V$ is  one of the suture components of $ W_1$, and at the same time it is   the sole suture  of $W_0$. Performing the suture gluing operation of  $W_1$  and $W_0$  along the domain $\Sigma_1\subset\p V$ we  obtain  a  sutured cobordism  isomorphic  to $\wt W$, see Figure \ref{fig:drilling1}.
Because $W_0$ is a trivial cobordism,   $\Morse(W_0,\p_- W_0)=0$, while   according to Lemma \ref{lm:Morse estimates}(ii) we have
  $$\Morse(W_1,\p_- W_1)\leq \max(\Morse(W,\p_-W),\Morse(V)+1).$$ Hence, 
  using Lemma  \ref{lm:Morse estimates}(i) we conclude
 \begin{align*}
 &\Morse(\wt W,\p\wt W)\leq \max(\Morse(W_1,\p W_1),\Morse(W_0,\p W_0))\\
 &\leq  \max(\Morse(W,\p_-W),\Morse(V )+1)=  \max(\Morse(W,\p_-W),\Morse(N )+1),
  \end{align*}
  as desired.
  \end{proof}
  \begin{lemma}\label{lm:subtr-handle}
Given positive  integers $m,k,\ell>0$ such that $m>k+\ell$  consider an index $m-k$ handle $H=D^{m-k} \times D^{k}$,  viewed as a sutured cobordism
 with $\p_-H=\p D^{m-k}\times D^{k}$ and $\p_+H= D^{m-k}\times \p D^{k}$. Let $D^\ell\subset D^{m-k}$ be the  equatorial disk cut out by the first $\ell$  coordinates.
 Denote  $S:=\p D^{\ell}\left(\frac12\right)\times D^{k}$, and  let $V$ be a tubular  neighborhood of $S$ in $H$.  Here the notation $D^j(r)$ stands  for a $j$-dimensional disc of radius $r$ and we write $D^j$ for $D^j(1)$. 
 Consider the sutured cobordism $$(W:=\overline{H\setminus V }, \p_+W:=\overline{\p_+H\setminus V \cup  \p V\setminus \p_+H},\;\p_-W:=\p_-H).$$ 
 Then $$\Morse(W,\p_-W)=m-k-\ell.$$
  \end{lemma}
  \begin{figure}[h]
\begin{center}
 \includegraphics[scale=.6]{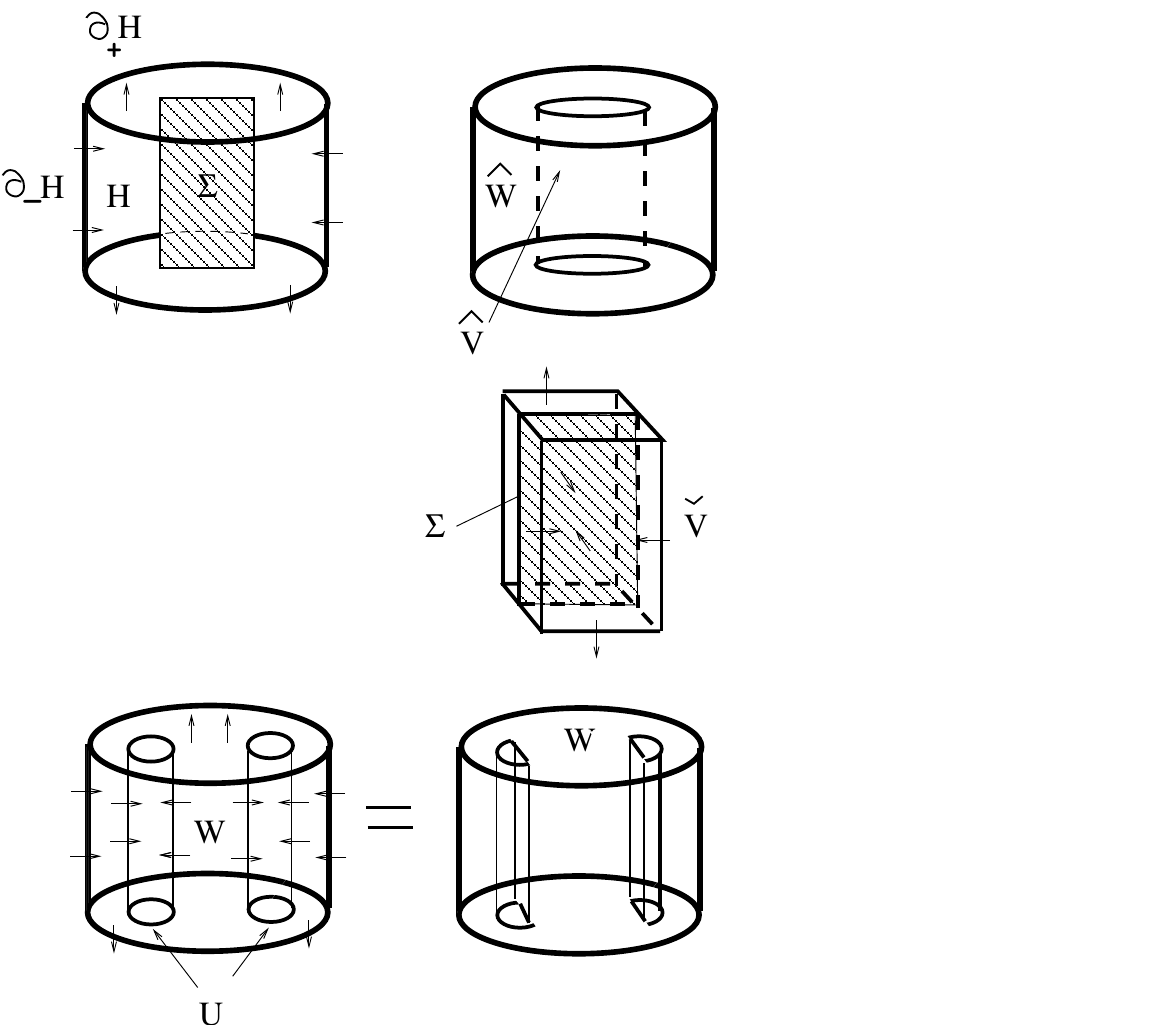} 
\caption{Two constructions of $W$; $m=3, k=\ell=1$.}
\label{fig:drilling2}
\end{center}  
\end{figure}
  \begin{proof}
  Let us give a different equivalent definition of the cobordism $W$.  Denote 
  $$\Sigma:=D^{\ell}\left(\frac12\right)\times D^{k},\;  \breve{ V}:=\Sigma\times D^{m-k-\ell}(\eps).$$  
  After smoothing non-essential corners we can view $\breve{V}$ as a sutured cobordism with $$\p_-\breve{V}:=D^{\ell}\left(\frac12\right)\times D^{k}\times \p D^{m-k-\ell}(\eps)\;\;\hbox{and} \;\;\p_+\breve{V}:=\p\left(D^{\ell}\left(\frac12\right)\times D^{k}\right)\times   D^{m-k-\ell}(\eps).$$
  
   Let $\wh V$ be the $\eps$-neighborhood  of the submanifold $\Sigma\subset H$.
  Consider a cobordism $\wh W$ obtained from $H$ by removing  $\wh V$:
  $$\wh W=\overline{H\setminus \wh V}, \p_+\wh W=\overline{\p_+H\setminus\wh V}\cup\overline{\p \wh V\setminus\p_+H},
  \p_-\wh W=\p_-H).$$  Note that $\p_-\breve{V}$ is contained in $\p_+\wh W$ and  by concatenating $\wh W$ and $\breve{ V}$   and  then smoothing non-essential corners we get a sutured cobordism diffeomorphic to $W$, see Fig.~\ref{fig:drilling2}.
 But $\wh V$ is equivalent to   a sutured version of the trivial cobordism over $\p_-H$, while $\breve{V}$ is just an index $m-k-\ell$ handle. Hence, $\Morse(W,\p_-W)= m-k-\ell$. 
    \end{proof}
   %%%%%%%%%%%%%%%%%
\subsection{Liouville and Weinstein sutured cobordisms}\label{sec:LW}
 A sutured cobordism $(W,Z)$ is called {\em Liouville} if it is  endowed with an exact symplectic form $\om=d\lambda$, and the vector field $Z$ is the Liouville vector field dual to the Liouville form $\lambda$.

 Any (sutured) Liouville cobordism can always be completed by attaching to 
 $\p_+W$ the positive part of the symplectization $(\p_+W\times[0,\infty),e^s(\lambda|_{\p_+W}))$, and to 
  $\p_-W$ the negative part of the symplectization $(\p_+W\times(-\infty,0],e^s(\lambda|_{\p_-W}))$, and matching the Liouville field $Z|_{\p_\pm W}$ with the vertical vector field $\frac{\p}{\p s}$, see Fig. \ref{fig:attach-ends}. The resulting
exact symplectic manifold (potentially with boundary) is denoted $\wt W$, and is called the {\em completion} of $W$.
   \begin{figure}[h]
\begin{center}
 \includegraphics[scale=.8]{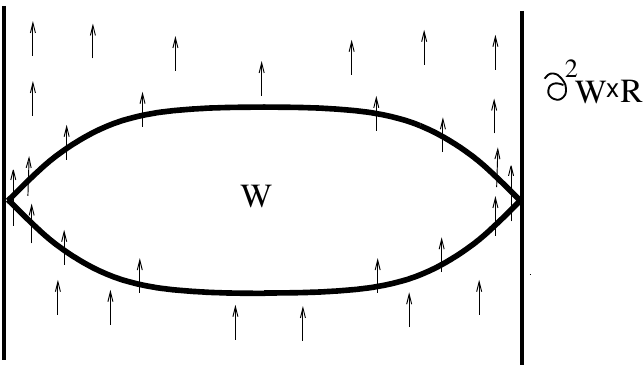} 
\caption{Attaching cylindrical ends.}
\label{fig:attach-ends}
\end{center}  
\end{figure}

By construction, the Liouville  vector field  $Z$ extends from $W$ to a complete Liouville field, still denoted by $Z$ on the completion    $\wh W\supset W$. In the case  when $\p^2W\neq\varnothing$ the manifold $\wh W$ has a boundary $\p^2W\times\R$ where each fiber $x\times\R$ is the complete $Z$-trajectory through $x\in\p^2W$. The completion construction yields the following useful statement.
  \begin{lemma}\label{lm:changing-form}
  Let $(W,\lambda)$ be a  sutured Liouville cobordism.
  Denote $\alpha_\pm:=\lambda|_{\p_\pm W}$. Then for any positive function $ {h}_+:\p_+W\to\R$ there exists   a Liouville cobordism structure $\wt\lambda$ on $W$ and a  function $ {h}_-:\p_-W\to\R$ which can be chosen  to agree with  $ {h}_+$  on $\Op\p^2W$,  such that $\wt\lambda|_{\p_\pm W}= {h}_\pm\alpha_\pm$.
  Similarly, for any positive function $ {h}_-:\p_-W\to\R$ there exists a    Liouville cobordism structure $\wt\lambda$ on $W$  and a function $ {h}_+:\p_+W\to\R$ which can be chosen equal to $ {h}_-$  on $\Op\p^2W$ such that $\wt\lambda|_{\p_\pm W}= {h}_\pm\alpha_\pm$. If the function   $ {h}_+$ (resp. $ {h}_-$)  is $\geq 1$ (resp. $\leq 1$) and $=1$ on $\p^2W$ then the function $ {h}_-$ (resp. ${h}_+$) can be chosen $\equiv 1$.
  \end{lemma}
  \begin{proof}
    Consider the completion $\wh W\supset W$.
For  any smooth   function $g:\p_\pm W\to\R$ denote by 
$\phi_\pm^g$ the embedding $\p_\pm W\to \wh W$ given by the formula
  $$\phi_\pm^g(x)=Z^{g(x)}(x),\;\; x\in\p_\pm W,$$
  where   $Z^t:\wh W\to \wh W$  is the time $t$ flow of the vector field $Z$.
  Note that $\left(\phi_\pm^g\right)^*\lambda=e^g\alpha_\pm.$
    Denote $\Sigma_\pm^g:=\phi_\pm^g(\p_\pm W).$
   
  Let  $g_+:=\ln h_+$, and  choose    a function $ g_-$ agreeing with $g_+$ on $\p^2W$, and  so that   the two  hypersurfaces   $\Sigma_+^{g_+}$ and $\Sigma_-^{ g_-}$ bound  in $\wh W$ a sutured cobordism $W^{g_+,g_-}$ (the second condition imposes a restriction on how large $g_-$ can be).  If these two conditions are satisfied, the embeddings $\phi_+^{g_+}$ and $\phi_-^{ g_-}$ extend to a diffeomorphism $
  G:W\to W^{g_+,g_-}$. Then the Liouville form $G^*\lambda$ restricts 
  to $\p_+ W$ as $ e^{g_+}\alpha_+$ and to $\p^-W$ as $e^{g_-}\alpha_-$, i.e. $h_-:=e^{g_-}$ is the required function.
  Similarly, setting $ g_-:=\ln h_-$ we can find  a function $ g_+$ equal to $  g_-$ on $\p^2W$ such that   the above claims hold,  i.e.  $\Sigma_+^{g_+}$ and $\Sigma_-^{ g_-}$ bound  in $\wh W$ a sutured cobordism $W^{g_+,g_-}$; then $h_+:=e^{g_+}$ is the required function. If $g_+\geq 0$ and  $g_+|_{\p^2W}=0$ then one can choose $ g_-\equiv 0$, and similarly if $  g_-\leq 0$ and $ g_-|_{\p^2W}=0$ then one can choose $g_+\equiv 0$.  

 \end{proof}  
 
 A (sutured) Liouville cobordism $(W,\om=d\lambda)$ is called {\em Weinstein} if the corresponding Liouville vector field $Z$ is gradient-like for a Morse function $\phi:W\to\R$.
 In the Weinstein case the critical points of the function $\phi$ have index $\leq n=\frac12\dim W$, see \cite{EliGro91}, and hence $\Morse(W,\p_-W)\leq n$.

Let us now consider the Liouville and Weinstein versions of the operations on sutured cobordisms introduced above  in Section \ref{sec:sutured}.

To define the  {\em concatenation} of two Liouville  cobordisms $(W,\lambda), (W',\lambda')$ choose a component $P$   of $\p_-W'$ and   a contact  form preserving embedding $$h:(P,\lambda'|_P),\to (\Int\p_+W,\lambda|_{\p_+W}).$$
This allows us to  match  the corresponding Liouville forms (as two Liouville forms on $\Op\p_-W'$ which coincide on $\p_-W'$ are isotopic) and hence define a Liouville cobordism structure on the result of the concatenation operation. Corner smoothing at non-essential corners, as defined above, uses the Liouville vector field $Z$, and hence will keep the boundary faces contact.
Similarly,  one defines the suture gluing operation of two sutured Liouville cobordisms if the matching diffeomorphism $g:U\to U'$ is chosen to preserve the induced forms:
$g^*\left(\lambda|_U'\right)=\lambda|_U.$

  The boundary inversion operation is more delicate in the Liouville context. It can be performed only for a restricted class of domains $U\subset\p_+W$, e.g. neighborhoods of Legendrian submanifolds, or, more generally, neighborhoods of Liouville hypersurfaces, see \cite{El-revis} .
  
 The   concatenation, gluing and inversion operations can also be  performed in  the category of sutured {\em Weinstein} cobordisms.
 \begin{remark}
When considering sutured Liouville and Weinstein cobordisms, or equivalently Liouville and Weinstein cobordisms between contact manifolds with boundary, it is customary to impose some constrains on the boundaries of the contact manifolds, e.g. contact convexity, in order to enable the application of methods from  the theory of holomorphic curves.   However,   the main results of this paper hold without any such constraints.  
 \end{remark} 
  \subsection{Attaching and subtracting Weinstein handles}\label{sec:surgery}

Let us recall some basic definitions and statements from the Weinstein surgery theory. For more detail, see 
\cite{CieEli12, Eli90, Wei91}.

Let $(\R^{2n},\sum\limits_{j=1}^n dp_j\wedge dq_j)$ be the standard symplectic space.
A {\em Weinstein handle} $H_k$ of index $k\leq n$ is
the domain
$$H_k=\left\{\sum\limits_{i=1}^k q_i^2\leq 1, \sum\limits_{i=k+1}^n q_i^2+\sum\limits_{i=1}^n p_i^2\leq 1 \right\}\subset\R^{2n}$$ endowed with the Liouville form
$$\lambda_k= \sum\limits_{i=1}^k (2p_idq_i+q_idp_i)+\sum\limits_{i=k+1}^n p_idq_i.$$
We have $d\lambda_k=\sum\limits_1^n dp_i\wedge dq_i$.

 The isotropic $k$-disc $$D:=\{q_{k+1}=\dots=q_n=0,\; p_1=\dots=p_n=0\}\subset H_k$$ is called  the {\em core }disc of the handle and the coisotropic $(2n-k)$-disc $$C:=
 \{q_{1}=\dots=q_k=0\}\subset H_k$$ the {\em  co-core} of the handle.
 We have a canonical decomposition $H_k=D\times C$.
 Denote $\Lambda_-:=\p D$, $\Lambda_+:=\p C$. 
 The handle $H_k$ can be viewed as a Weinstein sutured cobordism with
 \begin{align*}&\p_-H_k:= \Lambda_-\times C=\left\{\sum\limits_{i=1}^k q_i^2=1\right\}\cap H_k,\\
& \p_+H_k:=\Lambda_+\times D= \left\{\sum\limits_{i=k+1}^n q_i^2+\sum\limits_{i=1}^n p_i^2= 1\right\}\cap H_k.\end{align*}
   Indeed,   the Liouville vector field $$Z_k=\sum\limits_{i=1}^k \left(2 p_i\frac{\p}{\p p_i}-q_i\frac{\p}{\p q_i}\right)+\sum\limits_{i=k+1}^n p_i\frac{\p}{\p p_i}$$ is inwardly transverse to $\p_-H_k$ and outwardly transverse to $\p_+H_k$. 
   The function $ \sum\limits_{k+1}^n q_j^2+\sum\limits_1^n p_i^2-\sum\limits_1^k q_j^2$ serves as a Lyapunov function for the Liouville vector field $Z_k$.

 We will denote by $\xi_{k,\pm}$ the contact structure on $\p_\pm H^k $   defined by  the contact form $\lambda_{k,\pm}:=\lambda_k|_{\p_\pm H_k}$.

Let  $(W,\lambda)$  be a (sutured) Liouville cobordism, and suppose
 $h_-:(\p_-H_k,\lambda_{k,-})\to (\p_+W,\lambda|_{\p_+W})$ is a contact form preserving     embedding. Then we can attach the handle $H_k$  to $W$ using $h_-$ by concatenating   $W$ and $H_k$, as described in the previous section. We say that the resulting  sutured Liouville cobordism $(\wt W,\wt\lambda)\supset (W,\lambda)$ is obtained from $W$ by attaching an index $k$ Weinstein handle.
  Similarly,   given a contact form preserving embedding  $h_+:(H^k_+, \lambda_{k,+})\to (\p_-W,\lambda|_{\p_-W})$, we may consider   the backward concatenation  of    $H_k$  and   $W$.
     
   The effect of a Weinstein    handle attachment to the  positive  contact  boundary  $\p_+W$ is called the {\em direct}  Weinstein surgery of index $k$, while the effect of a  backward  attachment  to the negative   contact boundary $\p_-W$ is called the {\em inverse }   Weinstein surgery of index $2n-k$.\footnote{The direct and inverse Weinstein surgeries are also called $(-1)$- or $(+1)$-surgeries, respectively.}

   On the other hand,  if we are given a Liouville embedding  $G:(H_k,\p_-H_k;\lambda_k)\to (W, \p_-W;\lambda)$  (resp. $G:(H_k,\p_+H_k;\lambda_k)\to (W, \p_+W;\lambda)$) of the {\em whole} handle,  then one can {\em subtract} the handle, thus obtaining  a new Liouville cobordism $$(\wt W:=\overline{W\setminus G(H_k)},\wt\lambda=\lambda|_{\wt W})\subset (W,\lambda)$$ such that \begin{align*}
   &\p_+\wt W:=\p_+ W \;\hbox{ and}\; 
  \p_-\wt W:= \overline{\p_-W\setminus G(\p_-H_k)}\cup G(\p_+H_k),\\
  & 
   (\hbox{resp.}\;  \p_-\wt W:=\p_- W\;\hbox{ and}\;
   \p_+\wt W:= \overline{\p_+W\setminus G(\p_+H_k)}\cup G(\p_-H_k).)
   \end{align*}
   In fact, the result of this operation is   determined  (up to isomorphism) by the isotropic embedding $G|_D$, where $D\subset H_k$ is the core disc (resp. the co-isotropic embedding $G|_C$ where $C\subset H_k$ is the co-core disc).   
   
   %Similarly,  if we are given a  Liouville embedding   of the whole handle   $G:(W_k,H^k_+;\lambda_k)\to (W, \p_+W;\lambda)$ then one can {\em subtract} the handle, thus obtaining  a new Liouville cobordism $(\wt W,\wt\lambda)\subset (W,\lambda)$ such that $(\p_-W=\p_-\wt W)$ and $(\wt X:= W\setminus \Int \wt W,\wt\lambda|_X)$ is an elementary Weinstein cobordism  between $\p_+\wt W$ and $\p_+  W$  with a single handle of index $k$. This operation is   determined by the coisotropic embedding $G|_C$, where $C\subset H_k$ is the co-core disc.  

  The following lemma is a corollary of the classification of overtwisted contact structures established in
  \cite{BEM}, see Theorem \ref{thm:ot-properties} in the current paper.
  \begin{lemma}\label{lm:ot-attaching}
  Let $(Y,\xi)$ be an overtwisted contact manifold.
  Then any almost contact  embedding $h_\pm: (\p_\pm H_k,\xi_{k,\pm})\to (Y,\xi)$ is isotopic to a genuine contact embedding. We can furthermore ensure that the complement of the image $h_\pm(\p_\pm H_k)$ is overtwisted.
  \end{lemma}
%Added
 We recall that  given two contact manifolds $(M,\xi)$ and $(N,\eta)$, an {\em almost contact embedding}
is a pair $(f,\Phi_t)$ where $f:M\to N$ is a smooth embedding and $\Phi_t:TM\to TN$ is a  homotopy of  injective homomorphisms covering $f$ such that $\Phi_0=df$, $\Phi_1(\xi)\subset\eta$ and $\Phi_1|_\xi:\xi\to\eta$ is a fiberwise conformal symplectic homomorphism with respect to the    symplectic structures  induced  on $\xi$ and $\eta$ by any choice of contact forms.   A genuine contact embedding $f:M\to N$ can be viewed as   almost contact by adding a homotopy $\Phi_t:=df$.
  
   \begin{proof}[Proof of Lemma \ref{lm:ot-attaching}] We will  prove below a   more   
   general statement: {\em any almost contact embedding $(f,\Phi_t):(M,\xi)\to (N,\eta)$  between two contact manifolds, where  $\eta$ is overtwisted, is isotopic through the space of almost contact embeddings to a genuine contact embedding with an overtwisted complement.}   Note that the  normal bundle  $\nu$  to $\Phi_1(TM)$ in $TN$ has a   symplectic  structure $\mod\, C^{\infty}(M,\R)$, because $\nu$ is isomorphic to  the quotient  $\nu=\eta/\Phi_1(\xi)$ of two   symplectic bundles $\mod\, C^{\infty}(M,\R)$.
  The homotopy $\Phi_t$ allows us to pull back this structure to the normal bundle $\wt\nu$ to $f(M)$ in $N$.
A neighborhood $U$ of  the $0$-section $M$ in the total space $E$ of the bundle $f^*\wt\nu$ over $M$ has a canonical, up to  fixed on $M$ isotopy, contact structure $\wt\xi$ by Darboux (-Givental)'s theorem.
Having chosen such a neighborhood $U$, let us extend $f$ to a diffeomorphism $F$ of  $U$  onto a  tubular neighborhood of $f(M)$ in $N$.  We can extend the homotopy $\Phi_t$ to a homotopy $\wt\Phi_t$ of fiberwise isomorphisms
$TU\to TN$, covering $F$ and such that $\wt\Phi_0=dF$,  $\wt\Phi_1(\wt\xi)=\eta$ and  $\wt \Phi_1|_{\wt \xi}:\wt \xi\to\eta$ is a conformally symplectic isomorphism.   The push-forward $ (\wt\Phi_t)_*\xi$ is a homotopy of almost contact structures on $F(U)$ connecting $F_*\wt\xi$ with $\eta$. This homotopy  can be extended to a homotopy  $\eta_t$ of almost contact structures on all of  $N$ connecting $\eta_1=\eta$ with an almost contact structure $\eta_0$ which coincides with the genuine contact structure $  F_*\wt\xi$ on $F(U)$. Hence, by applying the classification of overtwisted contact structures, we get a genuine overtwisted contact structure $\wt\eta$ on $N$ extending $  F_*\wt\xi$, which is in the same  homotopy class of almost contact structures as $\eta$. Therefore, according to Theorem \ref{thm:ot-properties} $\wt\eta$  is also in  the same     homotopy class of genuine contact structures  and moreover, we can ensure that the contact structure $\wt\eta$ is overtwisted on every component of the complement  $N\setminus F(U)$. By  Gray's theorem \cite{Gray} there is a  diffeotopy $\phi_t:N\to N$ such that $\phi_0=\Id$ and $(\phi_1)_ *F_*\wt\xi=\eta$. Then $\phi_1\circ f$ is the required contact embedding $(M,\xi)\to (N,\eta)$.
   \end{proof}
 %End insertion %%%%%%%%%%%%%%%%%%%%%%%%
 %%%%%%%%%%%%%%%%%%%%%%%%
   
 %Given a manifold $W$ with boundary and a submanifold $(\Sigma,\p\Sigma)\subset (W,\p W)$ of positive codimension we say that  the relative Morse type of $(W,\Sigma)$ is $\leq k$ if  the  relative Morse type of  $(\wt W:=W\setminus \Int N, \p_-\wt W:=\p N\setminus \Int (N\cap\p W)$  is  $\leq k$, where  we denoted by $N$   a tubular neighborhood of $\Sigma$, .
 
  The following theorem is a corollary of  Lemma \ref{lm:ot-attaching} and  the results from  \cite{Eli90} (see also \cite{CieEli12}) which were proved using Weinstein handlebody constructions:
    \begin{thm}\label{thm:Weinstein-case}
Consider a  $2n$-dimensional connected almost symplectic sutured cobordism $(W,\eta)$.  Let $\lambda$ be a Liouville form on $\Op\p_-W$ such that the corresponding Liouville field $Z$ is inward transverse to $\p_-W$ and outward transverse to $\p_+W\cap\Op\p_-W$.   Suppose that  the contact structure $\xi_-=\{\lambda|_{\p_-W}=0\}$ is  compatible  with $\eta$. If $n=2$ we assume that $\xi_-$ is overtwisted on at least one of the components of $\p_-W$. Suppose  that $\Morse(W,\p_- W) \leq n$. Then $\lambda$ extends to  $W$  as  a sutured Liouville  cobordism  structure  with $(\p_+W,\Ker(\lambda|_{\p_+W}))$ and $(\p_-W,\Ker(\lambda|_{\p_+W}))$ as  its positive and negative   contact boundaries, and     such that  $d\lambda$ and  $\eta$ are homotopic  rel. $\p_-W$ as almost symplectic forms.    If   $\xi_-$ is overtwisted on at least one of the components of $\p_-W$, then one can arrange that $\xi_+$ is overtwisted as well (and in this case $\xi_+$ is uniquely determined up to isotopy by the homotopy class of the almost symplectic structure $\eta$).
  \end{thm}

  \begin{remark}\label{rem:Liouville-Weinstein}
  1. Connectedness of $W$ and the assumption $\Morse(W,\p_- W)\leq n$ imply that $\p_+W$ is connected.
  
  2. In fact, the construction in \cite{Eli90} and  \cite{CieEli12} yields a Weinstein  (and  not just Liouville) cobordism structure on $W$, but this result will   not be needed for our purposes.  Though the corresponding results are formulated in \cite{Eli90} and  \cite{CieEli12}  for non-sutured cobordisms,  the proof is local near the attaching spheres of the handles, and hence it works without any changes for the case when $\p^2W\neq\varnothing$.
  \end{remark}
   Theorem \ref{thm:Weinstein-case} together with Lemma \ref{lm:changing-form} implies the following   
\begin{cor}\label{cor:Weinstein-form}
Consider a  $2n$-dimensional  connected almost symplectic sutured cobordism $(W,\eta)$. Suppose  that   $\Morse(W,\p_- W) \leq n$.  Let $\lambda$ be a Liouville form on $\Op\p W$ such that the corresponding Liouville field $Z$ is inwardly transverse to $\p_-W$ and outwardly transverse to $\p_+W$. Suppose that  $\eta|_{\Op\p W}=d\lambda$ and that
  both contact forms $\alpha_\pm:=\lambda|_{\p_\pm W}$ are overtwisted (if $\p_-W$ is disconnected then we suppose that  at least one of its components is overtwisted). Then there exists a Liouville form $\Lambda$ on $W$ such that 
\begin{enumerate}
\item $d\Lambda$ is homotopic to $\eta$ via a homotopy of almost symplectic forms fixed on $\p_+W$;
\item $\Lambda|_{\p_+W}=\alpha_+$;
\item $\Lambda|_{\p_-W}= h\alpha_-$ for a function $h:\p_-W\to(0,1]$ which is equal to $1$ near $\p(\p_-W)=\p^2W.$

\end{enumerate}
\end{cor}
 
   The next proposition concerning an inverse Weinstein surgery on loose Legendrian knots is proven in \cite{CaMuPr} (see  \cite{Mur11} for a definition of loose Legendrians). For the convenience of the reader we provide a modified proof here.

\begin{prop}\label{prop:anti-surgery}
The inverse surgery on a loose Legendrian knot  produces an overtwisted contact manifold. 
\end{prop}

\begin{proof} 
Let $\Lambda\subset (Y,\xi)$ be a loose Legendrian sphere. Note that for an arbitrary neighborhood of the $0$-section  $U\subset T^*\Lambda$   the  inclusion $\Lambda \hookrightarrow Y$ extends to a contact  embedding  $(U\times[-1,1], dz-\lambda_\std)\to Y$.   Present the sphere $\Lambda$ as a union $A\cup B:=S^1\times D^{n-2}\cup D^2\times S^{n-3}$. Then  $T^*\Lambda=T^*A\cup T^*B$ and $T^*A=T^*S^1\times T^*D^{n-2}$. Hence, we can choose 
  the above neighborhood $U$ to contain   the product of arbitrary large neighborhoods 
 $U_1$ and $U_2$ of the $0$-sections in $T^*S^1$  and $T^*D^{n-2}$. We denote Liouville forms  in $T^*S^1$  and $T^*D^{n-2}$ by $u\,dt$ and $\sum\limits_1^{n-2}p_idq_i$, respectively.
 Let $\Gamma_\st$ be a Legendrian  stabilization of the $0$-section $\Gamma$ in the $3$-dimensional contact manifold $(V_1:=U_1\times [-1,1], dz-u\,dt)$.  By attaching an inverse $4$-dimensional handle $H_2$ along $\Gamma_\st$ we get a symplectic cobordism $X$ with $\p_+X = V_1$ and $\p_-X =: V_1^-$. The negative boundary $\p_-X$ is overtwisted, because a parallel copy of the zero section $\Gamma$ bounds a disk  in $V_1^-$ (after flowing from $\p_+X$ to $\p_-X$) and has $0$ Thurston-Bennequin number. According to \cite{Mur11}, there exists a Legendrian sphere $\Lambda_\st$ in a neighborhood  of $\Lambda$ which is Legendrian isotopic to  $\Lambda$ in $(Y,\xi)$ and so that $\Lambda_\st \cap V_1 \x U_2 = \Gamma_\st\times\{p=0\}$. Then the Liouville manifold  which we get by attaching the inverse handle $H_n$  along  $\Lambda_\st$  contains the product $X\times U_2$, and hence its negative contact boundary   contains $V_1^-\times U_2$. But the neighborhood of the $0$-section $U_2\subset T^*D^{n-2}$ can be chosen  arbitrarily large, and hence the resulting contact manifold is overtwisted by Theorem \ref{thm:ot-thick}.
 \end{proof}
 \begin{corollary}[see \cite{CaMuPr}]\label{cor:concordance}
 For any contact manifold $(Y,\xi)$ of dimension $>3$, there exists a Weinstein cobordism  structure
 on $W=Y\times [0,1] $ between the contact structure $\xi$ on $\p_+W:=Y\times 1$ and an overtwisted contact structure $\xi_{\rm ot}$ on $\p_-W:=Y\times 0$.
 \end{corollary}
 \begin{proof}
 On the trivial  Weinstein cobordism $W=Y\times[0,1]$  deform the Weinstein structure to create two critical points $a,b$  of index $n-1$ and $n$, respectively. Let $(Z, \xi_Z)$ be the intermediate level set for this cobordism.  Let $\Gamma\subset Y\times 1$ be the unstable  Legendrian  sphere for the critical point $b$. Consider a stabilization $\wh\Gamma$   of $\Gamma$ which is formally isotopic to $\Gamma$. By attaching an inverse handle  $H_n$ to $(Y=Y\times 1,\xi)$ along $\wh \Gamma$ we construct a Liouville cobordism between $Z$ and $Y$ with an overtwisted contact structure $\xi_{\rm ot}$ on $Z$ in the formal class of $\xi_Z$.
 By Lemma \ref{lm:ot-attaching} we can find a  coisotropic embedding of an $(n+1)$-dimensional sphere  into $(Z,\xi_{\rm ot})$ which is in the formal class of the unstable sphere of the critical point $a$ in $(Z, \xi_Z)$, and such that its complement is still overtwisted.  Hence one can attach an  inverse handle $H_{n-1}$ to $Z$ to get a Weinstein structure on the cobordism $Y\times [0,1]$ with the contact structure $\xi$ on the positive end  and an overtwisted contact structure  on the negative end, as desired.
 \end{proof}

 \subsection{Almost symplectic structures on codimension 2 submanifolds}

\begin{lemma}\label{lm:almost-complex}
Let $\Sigma$ be a codimension $2$ connected oriented  submanifold of a $2n$-dimensional  almost symplectic manifold $(W,\eta)$. Then $\eta$ is homotopic to $\wt \eta$ for  which $\Sigma$ is $\wt \eta$-symplectic   in the complement of a $(2n-2)$-dimensional ball $D\subset \Sigma$. If $W$ is a manifold with boundary and $(\Sigma,\p\Sigma)\subset(W,\p W)$ a submanifold with  non-empty boundary $\p\Sigma$, then $\eta$ is homotopic to $\wt \eta$ for  which $\Sigma$ is $\wt \eta$-symplectic everywhere. 
\end{lemma}
\begin{proof}
Choose an almost complex structure $J$ compatible with $\eta$. It is an equivalent problem to deform $J$ into an almost complex structure $\wt J$ for which $T\Sigma$ is $\wt J$-invariant, because then  for any  almost symplectic structure $\wt\eta$ compatible with $\wt J$ the submanifold $\Sigma$ will be almost symplectic.
We construct $\wt J$ inductively over cells of increasing dimension   in some cell-decomposition of $\Sigma$. If $\p\Sigma=\emptyset$ we assume that there is a unique $(2n-2)$-cell. If $\p\Sigma\neq\varnothing$ then $\Sigma$ can be isotoped into an arbitrarily small neighborhood of a $(2n-3)$-dimensional complex $C\subset \Sigma$, and hence it suffices to perturb $J$ to $\wt{J}$ near the cells of $C$.
Suppose that $\sigma\subset\Sigma $ is an $l$-dimensional cell, $l<2n-2$, and that we already have deformed $J$ to make $\Sigma$ $\wt J$-holomorphic near $\p\sigma$. Let us choose two vector fields $e_1,e_2$ which trivialize the co-oriented normal bundle to $T\Sigma\subset TW$ on $\Op\sigma$.
Furthermore, we can assume that $\wt Je_{1}=e_{2}$ on $\Op\p\sigma$.
We will arrange that  $\wt Je_{1}=e_{2}$ on $\Op \sigma$. There is a homotopy $e^t_2$, $t\in[0,1]$, over $\Op \sigma$ and fixed on $\Op\p\sigma$, such that $e^0_2=e_2$, $e^1_2=Je_1$, and so that $(e_{1},e_{2}^{t})$ remain linearly independent for all $t\in[0,1]$. Indeed, the obstruction to this lies in $\pi_l(S^{2n-2})=0$ for $l<2n-2$. Let $R^t$ be a covering homotopy of orientation preserving automorphisms of $TW$ such that $R^t(e_2)=e^t_2$ and $R^t(e_1)=e_1$ for all $t\in[0,1]$.
Let $J^t$, $t\in[0,1]$, be a homotopy of almost complex structures on $W$ such that 
\begin{description}
\item{-} $J^0=J$;
\item{-} $J^t=J$, $t\in[0,1]$, in a neighborhood of the $(l-1)$-skeleton;
\item{-} $J^t=R^t_*J$, $t\in[0,1]$, on $TW_{\Op\sigma}$.
\end{description}
Then $\Op\sigma$ is $J^1$-holomorphic. Continuing this induction construction over all cells $\sigma$ of dimension $\leq 2n-3$ we construct an almost complex structure $\wt J$ with the required properties.
  \end{proof}

 %%%%%%%%%%%%%%%%%%%
  \section{Proof of Theorem \ref{thm:cobord}}\label{sec:proofs}
  Corollary \ref{cor:concordance} implies that, if $n>2$, then it is sufficient to prove Theorem \ref{thm:cobord} in the case when the contact structure $\xi_+$ is overtwisted on each of the components of $\p_+W$, and for $n=2$ it is one of the assumptions that the contact structure $\xi_+$ is overtwisted on at least one of the components of $\p_+W$. Hence, Theorem \ref{thm:cobord} follows, even in a slightly stronger relative sutured cobordism version, from the following result.

\begin{thm}\label{thm:main-sutured}
Let $(W,\p_-W,\p_+W)$ be a sutured cobordism of dimension $2n\geq 4$ endowed with an almost symplectic structure $\eta$. Suppose that $\p_+W,\p_-W \neq\varnothing$. Let $\wt\Lambda$ be a Liouville form on $\Op\p W$ such that the corresponding Liouville vector field $Z$ is inwardly transverse to $\p_-W $ and outwardly transverse to $\p_+W $. Denote $\xi_\pm:=\Ker(\wt\Lambda|_{\p_\pm W})$.
Suppose that
\begin{itemize}
\item[-] $\eta|_{\Op\p W}=d\wt\Lambda$;
\item[-] 
  $\xi_-$ is overtwisted on at least one of the components of $\p_- W$; 
  \item[-]  if $n>2$   then $\xi_+$  is overtwisted on all of the  connected components of $\p_+W$, and if $n=2$ then $\xi_+$  is overtwisted  on at least one of the connected components of $\p_+W$.
  \end{itemize}
   Then there exists a Liouville cobordism structure $\Lambda$  on $W$ such that
   \begin{itemize}
\item  $\Ker(\Lambda|_{\p_\pm W})=\xi_\pm$;
\item $d\Lambda$ is homotopic to $\eta$ through a homotopy of almost symplectic forms and the homotopy can be taken to be fixed on $\p W$.
\end{itemize}
\end{thm}

  \subsection{Reduction to handles}\label{sec:red-to-elem}

We begin the proof of Theorem \ref{thm:main-sutured} with the following observation.
 
\begin{lemma}\label{lm:red-to-handle} Consider  a cobordism  $W$  obtained  by concatenating  two sutured  cobordisms  $W_0$ and   $W_1$. Suppose that Theorem \ref{thm:main-sutured} holds for the cobordisms $W_0$ and $W_1$. Then it also holds for $W$ provided that $\wt\Lambda|_{\p_+W}$ is overtwisted on at least one of the components of $\p_+W$ which intersects $\p_+W_1$.\end{lemma}
 \begin{proof}
 Recall that $\p_+W=\p_+W_1\cup(\p_+W_0\setminus\p_-W_1)$.
 One can modify the forms $\wt\Lambda$ and $\eta$  by an isotopy supported in $\Op\p_+W$ to make $\wt\Lambda|_{\p_+ W_1}$ overtwisted on at least one of the components of $\p_+W_1$. Furthermore, we can deform the Liouville form $\wt\Lambda$ on $\Op(\p_-W_1\cap \p_+W_0)  $, so that it restricts to $\p_-W_1\cap \p_+W_0$ as an overtwisted contact form
 in the formal homotopy class of $\eta$, relative to $\p(\p_{+} W_0\setminus\p_-W_1)$. We can then successively apply  Theorem \ref{thm:main-sutured} to $W_0$ and to $W_1$ to get the required form  $\Lambda$ on $W$.                   
\end{proof}
  
Any cobordism $W$ with $\p_+W\neq\varnothing$ can be presented as a result of successive handle attachments with handles of index $k=1,\dots, 2n-1$, starting with $\p_{-}W$. If $\Morse(W,\p_-W)\leq n$ then Theorem \ref{thm:main-sutured} for $W$ follows from the Weinstein handlebody theory, see Theorem \ref{thm:Weinstein-case}.

In general, the handles can be concatenated in the increasing index order. In particular, $W$ can be obtained by concatenating  handles  $W_1,\dots, W_j$ of index $2n-1$ to a cobordism  $W'$ with connected  boundary $\p_+W'$. By rearranging these handles via handle slides, if necessary, we can assume that  the form $\wt\Lambda$ is overtwisted on at least one of the components of $\p_+W$ which intersects $\p_+W_i$   for each $i=1,\dots, j$. Hence,  Lemma \ref{lm:red-to-handle} implies that it is sufficient to prove Theorem \ref{thm:main-sutured}  for handles. Moreover, the Weinstein handlebody theory, see Theorem  \ref{thm:Weinstein-case} and Corollary \ref{cor:Weinstein-form}, further reduces our task to proving Theorem \ref{thm:main-sutured} for handles of index $>n$.
 
We note, however, that the construction below works for all handles of index $k>1$, and in a slightly modified way for $k=1$ as well.
  \begin{lemma}\label{lm:fixing-eta}
Let $(W,\Lambda)$ be a  sutured Liouville cobordism structure on a handle $W$ of index $k$, and let $\eta$ be an almost symplectic form which coincides with $d\Lambda$ on $\Op\p W$. Suppose that the contact form $ \Lambda|_{\p W}$ is overtwisted on at least one of the connected components of either $\p_-W$ or $\p_+W$.  Then $W$ admits a Liouville cobordism structure $\wt\Lambda$ such that $d\wt\Lambda$ and $\eta$ are homotopic {\em relative} to $\p W$. 
\end{lemma}

\begin{proof}
 The  space of almost symplectic structures is homotopy equivalent to  the space of almost complex structures, and hence, the difference of the relative homotopy classes of $d\Lambda$ and $\eta$ can be viewed as  an element   $a\in\pi_{2n}(SO(2n)/U(n))$ (recall that $(W,\p W)$ is homeomorphic to $(D^{2n},S^{2n-1})$). Considering $D^{2n}$ as $D^{2n-1}\times [0,1]$, the class $a$ can be viewed as a loop of $SO(2n)/U(n)$-valued functions on $D^{2n-1}$ which are constant on $\Op\p D^{2n-1}$. In turn, this loop can be viewed as a loop of almost contact structures on $D^{2n-1}$.
 
  Using Theorem \ref{thm:ot-properties}  we can realize this loop as a loop of overtwisted contact forms  $\alpha_t$ on $D^{2n-1}$,  $t\in[0,1]$, fixed on $\Op\p D^{2n-1}$.
For a sufficiently large $C$ the form  $\lambda:=e^{Ct}\alpha_t$  defines a Liouville cobordism structure on the trivial cobordism
   $D^{2n-1}\times[0,1]$. Let $P$ be a sutured version of this cobordism.
  By our assumptions,  the contact form $\Lambda$ is overtwisted on one of the boundary components $V\subset \p_{+} W$ or $V\subset \p_{-} W$. Let us assume for determinacy that $V\subset \p_+W$. Using Lemma \ref{lm:ot-attaching},
  we construct a contact embedding $h:(\p_-P,\Ker\lambda_{\p_-P})\to (V,\Ker(\Lambda|_V))$. Furthermore, using Lemma \ref{lm:changing-form} we can adjust the Liouville form $\Lambda$ on $W$ to have
  $h^*\Lambda=\alpha_0$. Hence, we can use $h$ to construct a cobordism  $(W',\Lambda')=(W,\Lambda)\mathop{\cup}\limits_h(P,\lambda)$ by concatenating $P$ and $W$. There is an isotopy of $g_t:W\to  W'$ fixed outside $\Op\p_+W\subset W'$, moving points along trajectories  of the Liouville field $Z'$ corresponding to the Liouville form $\Lambda'$, such that $g_0:W\hookrightarrow W'$ is the inclusion and $g_1(W)=W'$.
 Then the Liouville form $\wt\Lambda:= g_1^*\Lambda'$ on $W$ has the required properties (i.e.\ we have ``killed'' the previous obstruction $a\in \pi_{2n}(SO(2n)/U(n))$). The  proof for the case when an overtwisted component is contained in $\p_-W$   works in a  similar way with an inverse concatenation of $P$ and $W$ via a contact embedding $\p_+P\to\p_-W$. 
\end{proof}
%%%%%%%%%%%%%%%%%%%%%%%%%%%%%%%%%%%%%%%%%%%%%%%%%%%%%%%%%%%%%%%%%%

\subsection{Beginning of the construction}\label{sec:beginning}

In what follows, we will be proving Theorem \ref{thm:main-sutured} in the case when $W$ is a handle of index $k>1$, i.e.\ $W=D^k\times D^{2n-k}$, $\p_-W=\p D^k\times D^{2n-k}, \p_+W=D^k\times\p D^{2n-k}$.
Moreover, in view of Lemma \ref{lm:fixing-eta} it will be sufficient to construct the form $\Lambda$ without caring about the relative homotopy class of $d\Lambda$.
The proof  will be done by induction over  $n$. The case $n=1$ is trivial.  Suppose $n\geq 2$ and that the theorem is already established for cobordisms of dimension $<2n$.

Let $\xi_+$ be the contact structure $\Ker(\wt\Lambda|_{\p_+W})$ on $\p_+W=D^{2n-k}\times\p D^k$. Assuming that $D^k$ and $D^{2n-k}$ are discs of radius $1$,  consider the equatorial disk $\Delta:=D^{k-1}(\frac12)\subset D^k$ of radius $\frac12$ and the codimension $2$ submanifold $\Sigma:=\p\Delta\times D^{2n-k}\subset W$ with boundary $\p\Sigma=\p\Delta\times\p D^{2n-k}\subset \p_+W$. 
There exists a diffeotopy $h_t:W\to W$,  $t\in[0,1]$, such that $h$ is fixed on $\Op\p_{-}W$, $h_0=\Id$ and  $h_1(\p\Sigma)$ is a contact submanifold  of $(\p_+W,\xi_+)$ (which in the case $n=2$ means that
 $h_1(\p\Sigma)$ is transverse to contact structure $\xi_+$). If $n>2$ or if $n=2$ and the index $k$ of the handle is $<3$ then we can also arrange that $\xi_+$ is overtwisted on (every component of) the complement of $h_{1}(\p\Sigma)$.
Indeed, this follows from Theorem \ref{thm:ot-properties}  and Lemma \ref{lm:ot-attaching},
because $\wt\Lambda|_{\p_+W}$ is overtwisted on each of the components  of $\p_+W$. If $n=2$ and $k=3$ then we can isotope each that  component of $h_{1}(\p\Sigma)$  to a  curve  transverse to the contact structure $\Ker(\wt\Lambda|_{\p_{+}W})$. To simplify the notation we will now write $\Sigma$ for the deformed submanifold $h_1(\Sigma)$.

The normal bundle to $\Sigma$ in $W$ is trivial, and, according to the contact normal neighborhood theorem, there exists a splitting $\wt N=\Sigma\times D^2$  of a tubular neighborhood of $\Sigma$ such that, if we define $\wt M_\ext:=\wt N\cap\p_+W=\p\Sigma\times D^2$, then the contact structure $\xi_+$ admits a contact form $\alpha_+$ satisfying $\alpha_+|_{\wt M_\ext}=\delta+udt$, where $\delta$ is a contact form on $\p\Sigma$ and $(\sqrt{u},t)\in([0,1]\times\R/(2\pi\Z))$ are coordinates on $\wt N$ induced from polar coordinates on the factor $D^2$. If $n>2$, we can also arrange that $\xi_+$ is overtwisted on (every component of) the complement of $\wt M_{\ext}$.

\begin{lemma}\label{lm:ext-closed-form}
The form $dt$ extends   from $\wt N\setminus \Sigma$ to $W\setminus\Sigma$ as a closed $1$-form  $\tau$
   supported away from $\Op\p_-W $.
\end{lemma}

\begin{proof}
Note that  $(\Sigma,\p\Sigma)$ represents a trivial homology class in $H_{2n-2}(W,\p_+W)$. This implies that
   a circle $S:=x\times \p D^2, x\in\Sigma, $ represents a non-zero homology class  in $H_1(W\setminus\Sigma,\p_-W;\R)$. 
Indeed, if there existed a relative chain $C$ in $(W\setminus \Sigma,\p_{-}W)$ bounding $S$, then the union of $C$ with the normal fiber $x\times D^2$ would form a relative 2-cycle $\Gamma$ transversely intersecting $\Sigma$ at 1-point, and hence the Poincar\'e-dual cohomology class $P\Gamma\in H^{2n-2}(W,\p_+W)$ would evaluate non-trivially on $\Sigma$, contrary to the triviality of the homology class $[\Sigma]\in H_{2n-2}(W,\p_+W)$. Hence, there exists a cohomology class $a\in H^1(W\setminus\Sigma,\p_-W;\R)$ such that $a([S])=1$ and $a|_{\Sigma\times y}=0.$ Therefore, $a|_{\wt N}=[dt]$ and hence $dt$ extends as a closed 1-form $\tau:=a+d f$ for some smooth function $f$ vanishing on $\Op \p_{-} W$.
\end{proof}

 Using Lemma \ref{lm:almost-complex} we  deform the  almost symplectic structure $\eta$ to make $\eta|_\Sigma$ almost symplectic.

   \subsection{The case $n=2$}
  Consider first the case $n=2$.
For a sufficiently large constant $C$ the form  $C\alpha_+ -cdt$ is contact for all $c\in[0,1]$.
 Let us consider a smaller tubular neighborhood $N=\{u\leq\frac1{2C}\}\subset \wt N$. Denote $v=Cu$, so that we have $N=\{v\leq\frac 12\}$ and 
 $C\alpha_+|_N=C\delta+vdt$.
 To simplify the notation we rename $C\alpha_+$  and $C\delta$ back to $\alpha_+$ and $\delta$.

   If $n=2$ then $\Sigma$ is a $2$-dimensional surface with boundary, and hence,  the form $\delta$  can be extended as a Liouville form  $\lambda$ to $\Sigma$  such that  the corresponding Liouville vector field transverse to $\p\Sigma$ in the outward sense.
   Denote $M_\ext:=\p N\cap \p_+W,\; M_\intr=\p N\setminus M_\ext,$ 
   see Fig. \ref{fig:FigN}.

Consider the  Liouville $1$-form 
 $\mu:=(v-1)dt+ \lambda$     on $ N$.  We have $\mu|_{M_\ext}=(\alpha_+-\tau)|_{M_\ext}$ and the  Liouville field  
corresponding to $\mu$ points into $N$ along $M_\intr$, and out of $N$ along $M_\ext$.  Note  that while   $\mu$ blows up along $\Sigma$, the symplectic form $d\mu$ extends to $  N$ as a symplectic form, and the extended form coincides with $d\lambda$ over $\Sigma$.

\begin{figure}[h]
\begin{center}
 \includegraphics[scale=.38]{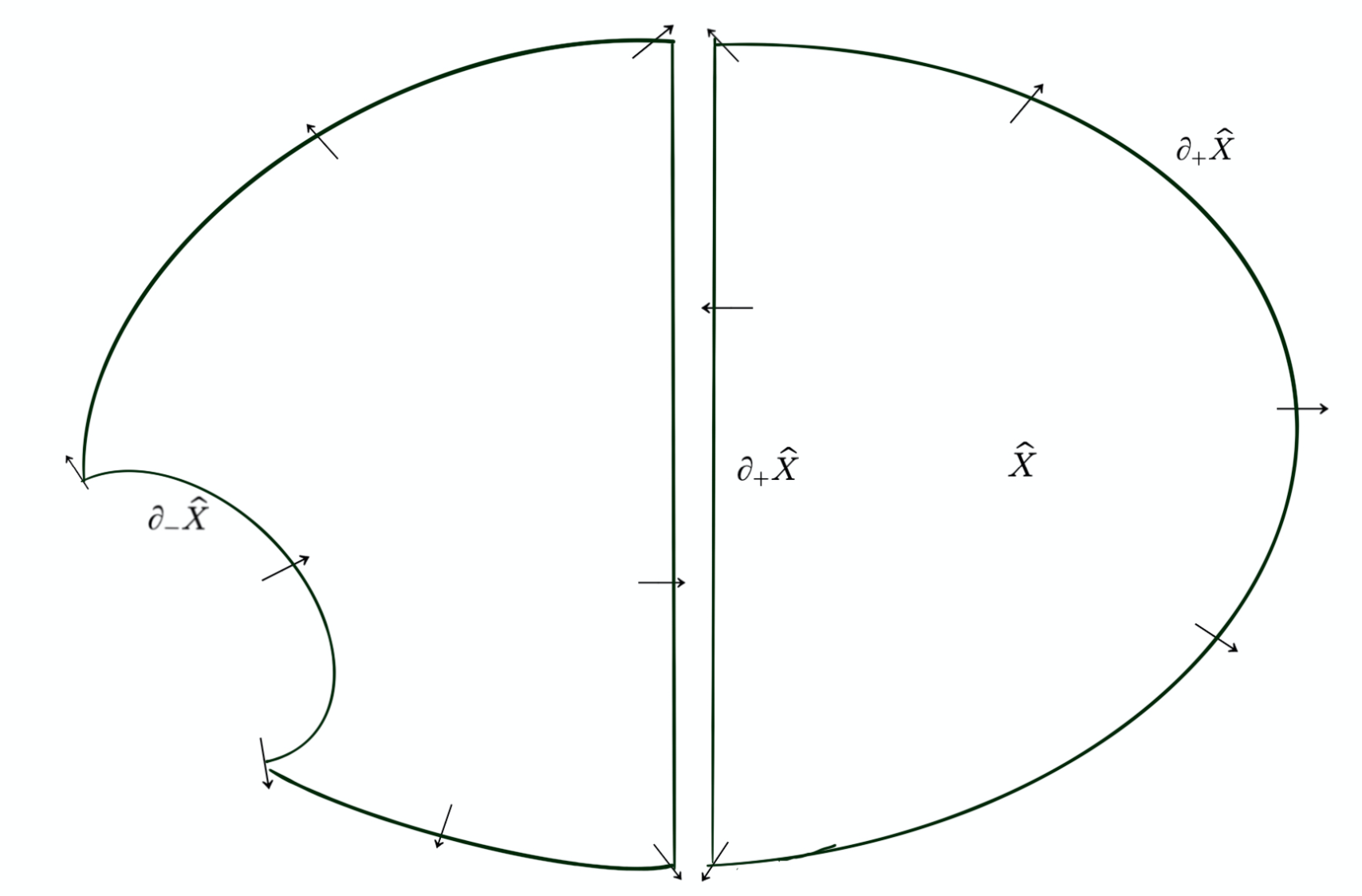} 
\caption{Cobordism $\wh X$}
\label{fig:cobwhX}
\end{center}  
\end{figure}

Consider the sutured cobordism $$\wh X=W\setminus N, \p_-\wh X=\p_-W,\p_+\wh X=(\p_+W\setminus M_\ext)\cup M_\intr,$$ see
Fig. \ref{fig:cobwhX}.

 Lemma \ref{lm:subtr-handle}, see Fig. \ref{fig:drilling2},  implies 
  that $\Morse(\wh X,\p_-\wh X)$ is equal to $1$.\footnote{In the original version of the paper the Morse type  of $(W\setminus\Sigma, \p_-W)$ was stated incorrectly. We thank the anonymous referee for pointing out our mistake.}   
  Consider  a Liouville form $\wh \Lambda$ on $\Op\p_+\wh X$ which restricts  as $\alpha_+-\tau$ to $\p_+W\setminus M_\ext$ and which equals to $\mu$ on $  M_\intr$.

By applying Theorem \ref{thm:Weinstein-case} we can extend $\wh\Lambda$ to    a Weinstein cobordism   form $\wh \Lambda$ on  $\wh X$ which  restricts as $f\alpha_-$ to $\p_-W$, and which is in the formal  rel. $\p_-W$  homotopy class of $\eta$.   We can then extend $\wh\Lambda$ to $W\setminus \Sigma$ as equal to $\mu$ on $N\setminus\Sigma$.
Then the form $\Lambda=\wh\Lambda+\tau$  is the required Liouville cobordism structure on $W$. Its relative to $\p W$ homotopy class can be fixed using Lemma  \ref{lm:fixing-eta}, and  this concludes the proof    of Theorem \ref{thm:main-sutured}  in the case $n=2$.
  
        %%%%%%%%%%%%%%                                                                                                                                                                                                                                                                                                                                                                                                                                                                                                                                                                                                                                                                                                                                                                                                                                                      
 \subsection{Case $n>2$}
  
\begin{figure}[h]
\begin{center}
 \includegraphics[scale=.50]{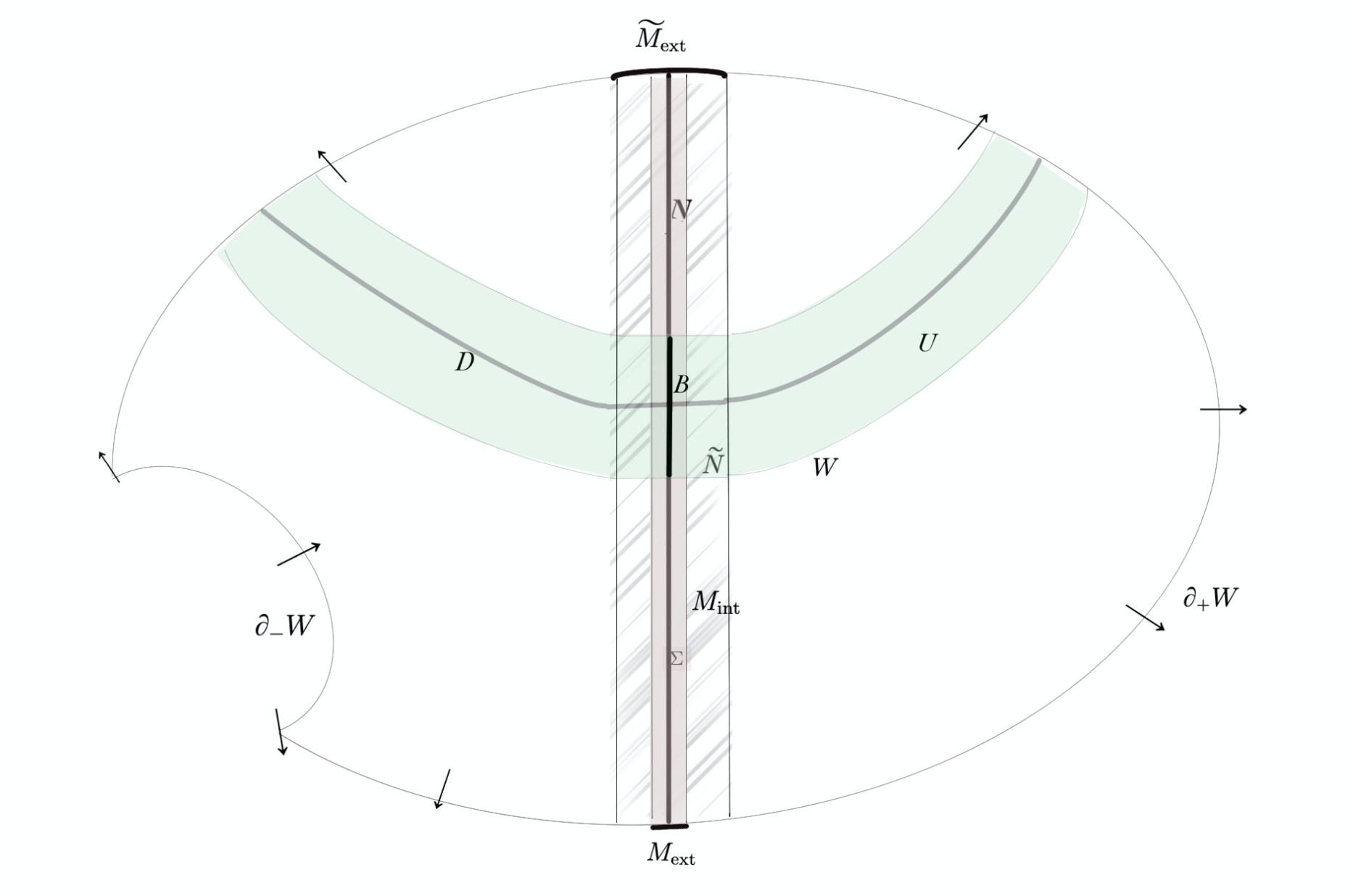} 
\caption{Hypersurface $\Sigma$, disc $D$, and their neighborhoods}
\label{fig:FigN}
\end{center}  
\end{figure}

   %%%%%%%%%%%%%%%%%%%%%%%
 Suppose $n> 2$.
     Choose an embedded ball $B\Subset\Int\Sigma$ centered at a point $p\in\Sigma$  and consider an overtwisted contact form $\beta$ on $\p B$ in the formal homotopy class determined by $\eta$.
 Let   $(D,\p D)\subset (W,\p W)$  be an embedded $2$-disc such that  
  $D\cap \wt N=p\times D^2$.    We extend the polar coordinates $(\sqrt{u},t)$ to the disc $D$ such that $u|_{\p D}=K$, where the constant $K>1$ will be chosen later.
  
   Let  $U =  D\times B   $ be  a split  tubular neighborhood of $D$ such   that $U\cap\Sigma=B$ where $B \subset\Sigma$ is the ball chosen above,  and the splittings of $U$ and $\wt N$ agree on $\wt N\cap U$, see Fig. \ref{fig:FigN}.
     We  also assume that   $B\times\p D\subset\p_+W$.  Choosing a radial  coordinate $s$   on $B$, we extend  the coordinates $s$ and  $(u,t)$ to $U$  by pulling them    back to $U=D\times B$ via  projections to the second and first factors, respectively.    
     For $\sigma\in(0,1]$ we denote $$U^\sigma:=
 \{s<\sigma\}, \;  H_\ext^\sigma=\p U_\sigma\cap\p_+W, \; H_\intr^\sigma:=\p U\setminus H_\ext^\sigma$$
 and  will write $H_\ext$ and $H_\intr$ instead of $H_\ext^1$ and $H^1_\intr$, see Fig. \ref{fig:InsideU}.

 The closed forms $\tau$ and $dt$ on $U$ are in the same cohomology class, and hence we can modify $\tau$ to  make it equal to $dt$ on $(N\cup U)\setminus \Sigma$.

  \begin{lemma}\label{lm:adding-closed-form2}
 There exists a contact form $\alpha$ on
 $\p_+W\setminus\p D$ with the following properties:
     \begin{enumerate}
  \item  $\alpha=C\alpha_+$ on $\wt M_\ext\cup\Op\p^2W$ for a constant $C\geq 1$; in particular, $\alpha|_{\wt M_+} = Cudt+ C\delta$;    
   \item the contact structures $\Ker\alpha$ and $\xi_+|_{\p_+W\setminus\p D}$ are homotopic relative  $\wt M_\ext\cup \Op\p^2W$;
   \item $ \alpha|_{   H^{\frac12}_\ext\setminus\p D} =2dt+s\beta$  
  \item $\alpha -c\tau$ is contact for any $c\in[0,1]$.
 \end{enumerate}
 \end{lemma}

 \begin{proof}
 Thanks to our choice of $\beta$ on $\p B$ in the formal homotopy class determined by $\eta$, we can use   Theorem \ref{thm:ot-properties} to    construct a contact form $\alpha'$ on $\p_+W\setminus \p D$ which satisfies  conditions (i) and (ii) and which is equal to $2dt+s\beta$ on $H_\ext\setminus\p D$.     By multiplying $\alpha' $  by a sufficiently large constant $C$  we can   ensure that  
 $C\alpha'-c\tau$ is contact for all $c\in[0,1]$.
 We claim that there exists a $C^\infty$-function $h:[0,1]\to[1,C]$ such that 
 \begin{itemize}
 \item[$\dagger$] $h(s)=1$ for $s\in[0,\frac12]$;
 \item[$\dagger$]   $h(s)=C$ near $s=1$ and
  \item[$\dagger$]     the form $h(s)(2dt+s\beta)-cdt$  is contact on $H_\ext\setminus \p D$ for all $c\in[0,1]$. 
  \end{itemize}
 Indeed,  let us first observe that it is sufficient to verify the latter condition for $c=1$.
 Indeed, let $R$ be the Reeb vector field of the form $h(s)(2dt+s\beta)$. The form $h(s)(2dt+s\beta)-cdt$ is contact if $cdt(R)\neq 1$ everywhere in $H_\ext\setminus \p D$. But on $H^{\frac12}_\ext\setminus \p D$ we have $R=\frac12\frac{\p}{\p t}$ and $dt(R)=\frac 12$. Hence, if  the form $h(s)(2dt+s\beta)-dt$ is contact everywhere then $dt(R)<1$, and thus, $cdt(R)<1$ as well.
The contact condition for the form $h(s)(2dt+s\beta)-dt=(2h(s)-1)dt+sh(s)\beta$ is equivalent to the inequality
 $\frac{d}{ds}\left(\frac{s h(s)}{2h(s)-1}\right)>0$.
Take a $C^\infty$-function $\psi:[0,1]\to\R$ such that $\psi(s)=s$ for $s\in[0,\frac12]$, $\psi(s)= \frac{Cs}{2C-1}$ near $s=1$,  and which satisfies inequalities $2\psi(s)>s$ and $\psi'(s)>0$ for all $s\in[0,1]$. Then the function
  $h(s)=\frac{\psi(s)}{2\psi(s)-s}$ has the  required properties. Indeed, we have $h(s)=1$ for $s\in[0,\frac12]$ and $h(s)=C$ near $s=1$. We also have   $\frac{s h(s)}{2h-1}=\psi(s)$, and hence the form $$ h(s)  (2dt+s\beta)-dt=(2h(s)-1) dt+\frac{s h(s)}{2h(s)-1}\beta =(2h(s)-1)(dt+\psi\beta)$$ is contact because $\psi'(s)>0$. Hence, 
  the form $$
  \alpha:=\begin{cases}
  h(s)(2dt+s\beta),&\hbox{on}\;H_\ext\setminus \p D;\\
    C\alpha',&\hbox{on}\; \p_+W\setminus H_{\ext}\\
    \end{cases}
    $$
has the required properties.
  \end{proof} 
  
If the constant $C$ in Lemma \ref{lm:adding-closed-form2} is chosen large enough then the induction hypothesis allows us to find a Liouville form   $\lambda$ on $\Sigma\setminus p$  with a conical singularity at $p$, and such that
\begin{itemize}
\item $\lambda|_{\p\Sigma}=C\delta$;
\item the Liouville field of 
$\lambda$ is outwardly transverse to $\p\Sigma$;
\item    $\lambda|_{B}=s\beta$ and 
\item $d\lambda|_{\Sigma\setminus B}$ is in the almost symplectic homotopy class of $\eta$. 
\end{itemize}

We now  fix the constant $K=u|_{\p D}$ to be equal to $2C$.
Denote \begin{align*}
&v:=Cu,\;\;  N:=\left\{v<\frac12\right\}\subset\wt  N,\;\;  M_\intr:=
\left\{v=\frac12\right\},\;\;
M_\ext=N\cap\p_+W,\\
&P :=\left\{\max\left(1-v, \frac{v-1}2\right)\leq s\leq \frac12\right\}.
\end{align*}
 We have $\p N=M_\intr\cup M_\ext$ and $H_\ext=\{v=2,s<1\}$.

 \begin{figure}[h]
\begin{center}
 \includegraphics[scale=.43]{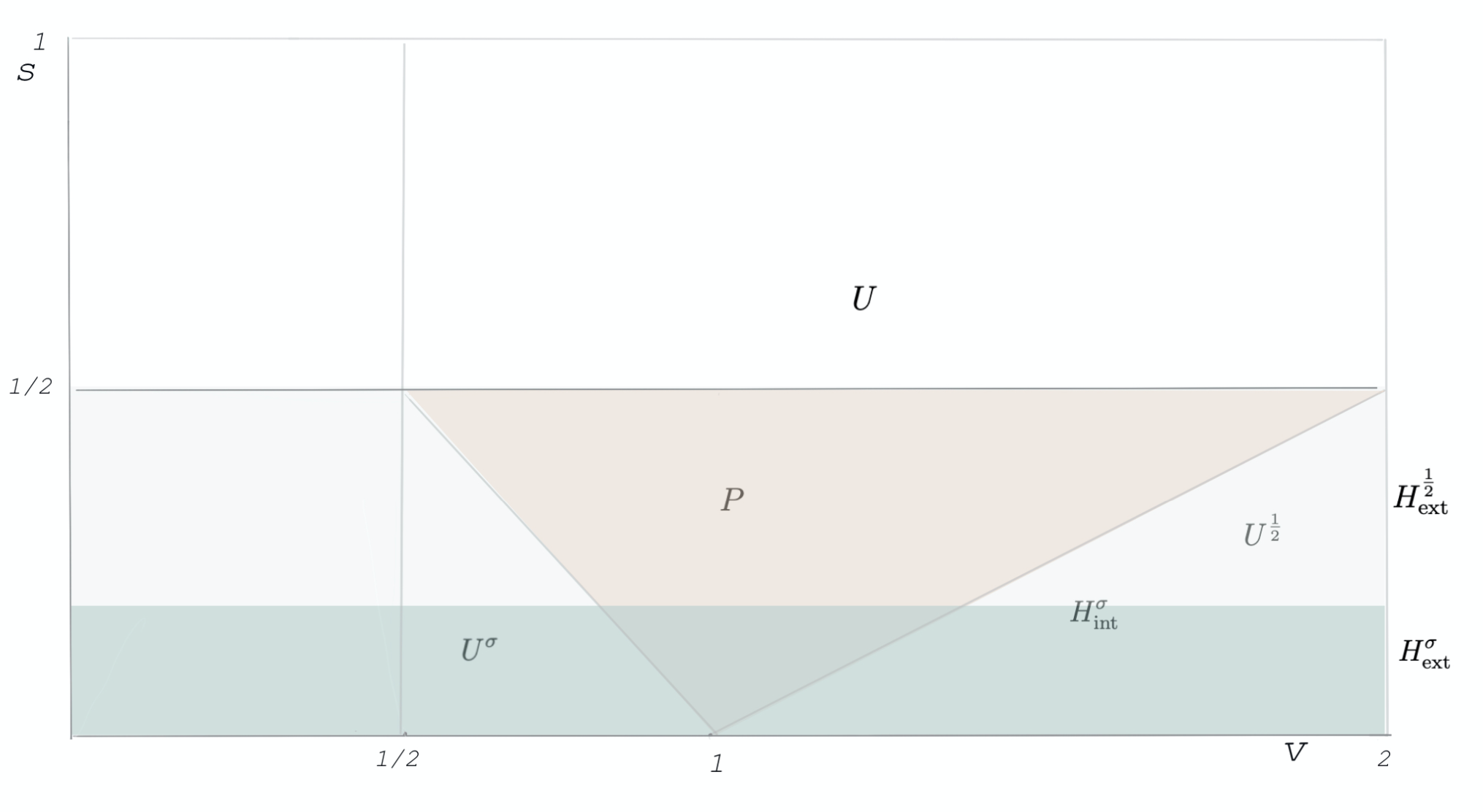} 
\caption{The inside of the neighborhood $U$}
\label{fig:InsideU}
\end{center}  
\end{figure}

 \begin{prop}\label{prop:main3}
For a sufficiently small $\sigma\in(0,\frac12)$ 
 the sutured cobordism $$(  X^\sigma:=W\setminus U^\sigma,\p_- X^\sigma:=\p_-W\cup H^\sigma_\intr ,\;\p_+  X^\sigma:=\p_+W\setminus H^\sigma_\ext)$$ 
 admits a Liouville cobordism structure $ \Lambda^\sigma$ such that
 \begin{itemize}
 \item$ \Lambda^\sigma|_{\Op H^\sigma_\intr}=s\beta+vdt;$
 \item $\Lambda^\sigma|_{\p_+X^\sigma}=\alpha$;
 \item $\Ker(\Lambda^\sigma|_{\p_-W})=\xi_-$. 
 \end{itemize}   \end{prop}
 
 Before proving this proposition we deduce from it Theorem  \ref{thm:main-sutured}.
 
\begin{figure}[h]
\begin{center}
 \includegraphics[scale=.50]{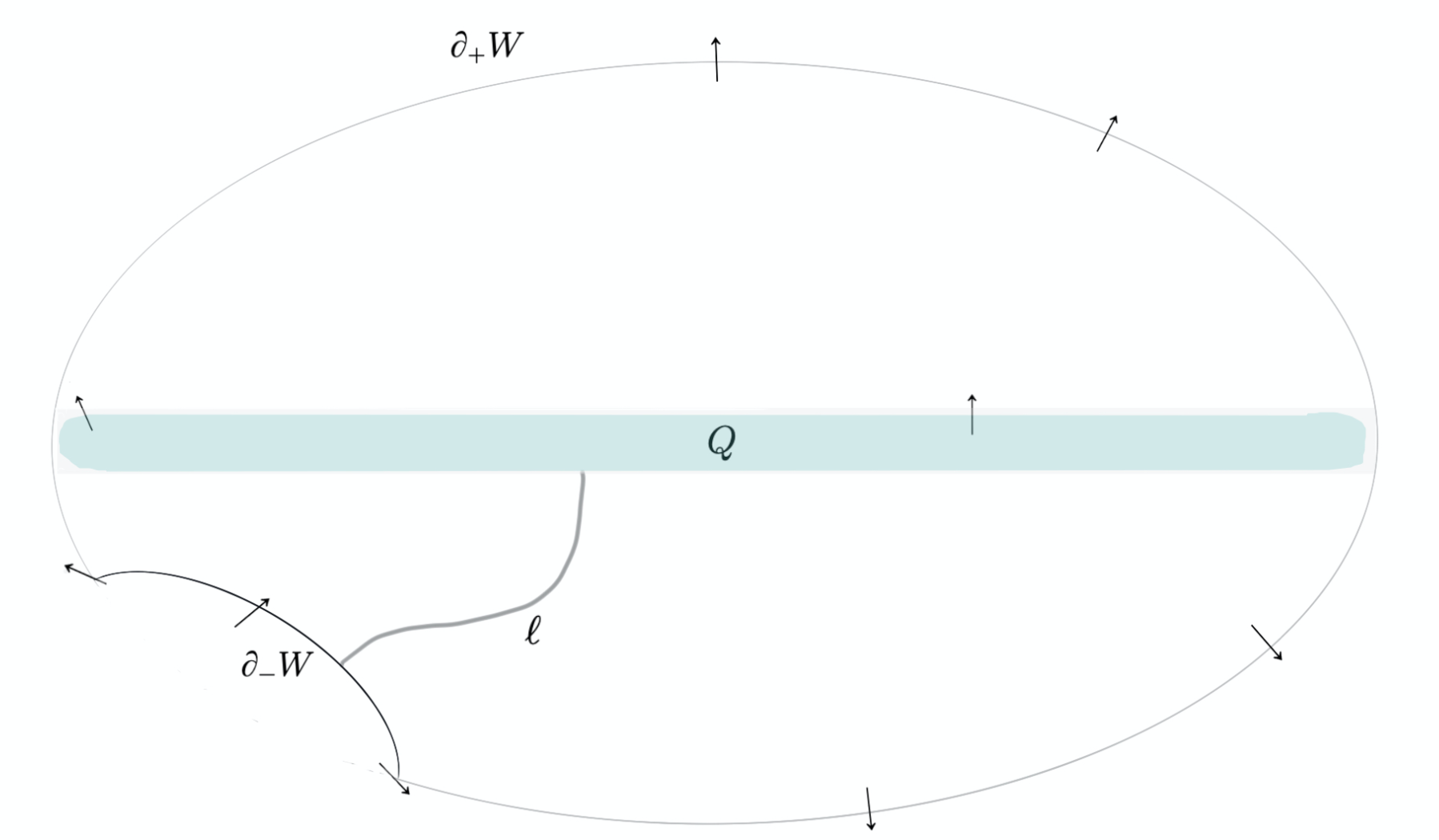} 
\caption{Connect summing of $\p_-W$ and $Q$}
\label{fig:connect-sum}
\end{center}  
\end{figure}

\begin{proof}[Deducing Theorem  \ref{thm:main-sutured} from Proposition \ref{prop:main3}]
Using Theorem \ref{thm:ot-properties} we can  extend   $\Lambda^\sigma$ to   a neighborhood  $\Omega\supset H^\sigma_\ext$ so that $ \Lambda^\sigma|_{\p_+W}$ is a contact form and the contact structure  $\Ker(\Lambda^\sigma|_{\p_+W})$ on    $\p_+W$  is 
 homotopic to the contact structure $\xi_+$.     Take an embedded $2n$-ball $Q\subset  \Int U^\sigma $ with boundary $\p Q\subset \p U^\sigma\cup\Omega$ transverse to the Liouville field of the form $\Lambda^\sigma$. The   extended this way form $\Lambda^\sigma$   defines a Liouville cobordism structure  on $W':=W\setminus \Int Q$ with $\p_-W'=\p_-W\cup \p Q, \p_+W'=\p_+W.$

   Our next goal
is   to   modify the cobordism  $W'$ by connect summing $\p_-W$ and $\p Q$, see Fig. \ref{fig:connect-sum}. To do that,  choose an   embedded  arc $\ell$  connecting     $\p_-W$  with   $\p Q$,  which coincides  near $\p\ell$ with a flow line  of  the Liouville field of $\Lambda^\sigma$ and such that $\int\limits_\ell\Lambda^\sigma=0$. We can modify the form   $\Lambda^\sigma$ by adding an exact form $dH$ supported in $\Int W$ to make it vanishing on $\ell$.
   By subtracting  from $W'$ the Weinstein handle $H_1$ with the core disc $\ell$, see  Section \ref{sec:surgery}, we get  a Liouville  cobordism $W''$ whose new negative boundary is the connected sum   $\p_-W\# 
   \p Q$ along $\ell$. Finally we note that there exists a fixed near $\p_+W$ isotopy $\rho_\tau:W'\to
  W$, $\tau\in[0,1]$,   such that $\rho_0=\Id$ and $\rho_1(W)=W''$. The pull-back Liouville form $\Lambda:=\rho_1^*(\Lambda^\sigma)$   is a Liouville cobordism  structure on $W$ which restricts to  $\p_\pm W$  as   an overtwisted contact form   in the formal homotopy class of $\xi_\pm$, and hence, in view of Theorem  \ref{thm:ot-properties}, the contact structure $\Ker(\Lambda|_{\p_\pm}) W$ is isotopic to $\xi_\pm$.  \end{proof}
  
    \subsection{Proof of Proposition \ref{prop:main3}}
  \begin{lemma}\label{lm:X-type}
   Let
   $(X,\p_-X ,\p_+X)$ be  the sutured cobordism
\begin{align*}
&X:=  W\setminus(U^{\sigma}\cup N),\\
 &\p_-X:=(H^{\frac12}_\intr\setminus N) \cup \p_-W,\;\;\p_+X:=    (\p_+W  \sm M_\ext)   \cup (M_\intr\setminus U^{frac12}).
      \end{align*}
   Then $\Morse( X,\p_-X)\leq 3$.
    \end{lemma}
    \begin{proof}
   As in the case $n=2$,  Lemma \ref{lm:subtr-handle} implies that   $\Morse(\wh X,\p_-\wh X)=1$  for   the  sutured cobordism $\wh X=W\setminus N, \p_-\wh X=\p_-W,\p_+\wh X=(\p_+W\setminus M_\ext)\cup M_\intr$,
  see Fig. \ref{fig:drilling2} and Fig. \ref{fig:cobX}.  
  
\begin{figure}[h]
\begin{center}
 \includegraphics[scale=.45]{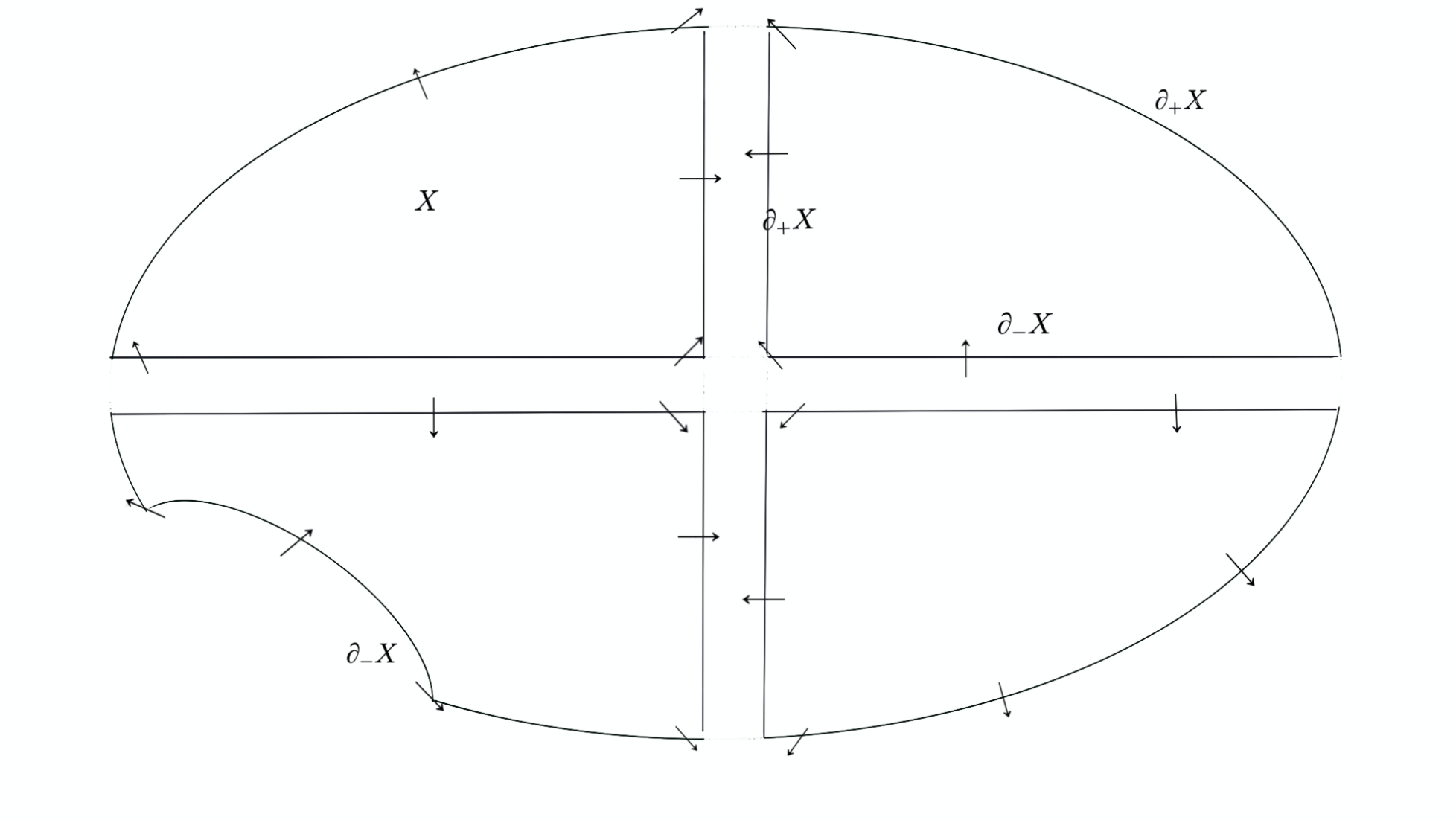} 
\caption{Cobordism X}
\label{fig:cobX}
\end{center}  
\end{figure} 
Hence     the pair $(\wh X,\p_+\wh X) $  is $(2n-2)$-connected.
Taking into account that $n\geq 3$ we conclude that the annulus 
$A=\{\frac12\leq v\leq 2,\,s=0\}\subset D$ is isotopic in $\wh X$ relative $\p A$ to an annulus  contained in  $\p_+\wh X$.  Moreover, there exists an embedding of the sutured version $\wh A$ of the trivial cobordism $A\times I$ into $\wh X$ such that $\p_-\wh A=A$ and $\p_+\wh A\subset \p_+\wh X$. Indeed, the general position argument provides us with an embedding if $n>3$ and an  immersion if $n=3$. But any self-intersection point can be pushed to the boundary, i.e. one can get rid of the self-intersection points by changing the embedding class of the boundary $\p_+\wh A\subset \p_+W.$

The cobordism  $(X, \p_-X,\p_+X)$ can be  obtained  from the cobordism  $(\wh X,\p_+\wh X,\p_-\wh X)$ by removing   a tubular neighborhood of $ A$ and adjoining the boundary of this neighborhood to $\p_+ X$, as it described in  Lemma \ref{lm:remove}. Hence, we can apply that lemma to  conclude that $$\Morse(X, \p_-X)\leq\max(\Morse(\wh X,\p_-\wh X),\Morse(A)+1)=     3.$$
    \end{proof}

   \begin{proof}[Proof of Proposition \ref{prop:main3}] 
 %%%%%%%%%%%%%%%%%%%%

%\begin{figure}[h]
%\begin{center}
%\includegraphics[scale=.35]{CobordW4}
%\caption{The structure of the cobordism $W$.}
%\label{fig:big}
%\end{center}  
%\end{figure}

 %\begin{figure}[h]
%\begin{center}
%\includegraphics[scale=.6]{CobXhat.pdf}
%\caption{The inside of the neighborhood $U$.}
%\label{fig:Xhat}
%\end{center}  
%\end{figure}  

Consider the form $\mu$ on $(U\setminus D)\cup(N \setminus  \Sigma)$ which is equal to $\lambda+(v-1)dt$ on $N\setminus\Sigma$ and equal to $s\beta+(v-1)dt$ on $U\setminus D$. 
Lemma \ref{lm:X-type} allows us to  apply  
  Corollary \ref{cor:Weinstein-form} to construct a  sutured Liouville cobordism structure  $\wh\Lambda$ on $(X,\p_-X,
\p_+X)$  such that       \begin{itemize}
    \item[-] $\wh \Lambda=\mu$ on $\Op(\p N\cap X)$;
    \item[-] $\Ker(\wh\Lambda|_{\p_-W})=\xi_-$;
    \item[-] $\wh\Lambda|_{\p_+W\setminus (M_\ext\cup H^{\frac12}_\ext)}=\alpha-\tau $;
    \item[-] $\wh\Lambda|_{H^{\frac12}_\intr}=\phi\mu$ for  a function  $\phi: H_\intr^{\frac12}\to(0,1]$    such that   $\phi|_{\Op\p H_\intr^{\frac12}}=1$.
    \end{itemize}
Note  that the Liouville vector field $Z$  corresponding to the form $\mu=(v-1)dt+s\beta$ on     $U\setminus D$   is equal to $s\frac{\p}{\p s}+(v-1)\frac{\p}{\p v}$,
 and therefore, its negative flow  flow  $Z^{\nu}, \nu\leq 0$  is given by the formula $$Z^\nu(v,s)=(1+e^{\nu}(v-1),e^{\nu} s), \;\;\nu\leq 0,, $$ and  hence, it
  leaves invariant  the cone 
 $P :=\left\{\max\left(1-u, \frac{u-1}2\right)\leq s\leq \frac12\right\}, $ see Fig.~\ref{fig:InsideU}.
 
 Define an isotopy
    $$\alpha_\nu: H^{\frac12}_\intr\to   P ,\;\nu\in [0,1],$$    by the formula 
 $$\alpha_\nu(x)=Z^{\nu\ln\phi(x)},\; x\in   H_\intr,\nu\in[0,1].$$
We have 
$\alpha_0=\Id $, $\alpha_1^*\mu=\phi\mu$ and 
   $\alpha_\nu(H_\intr)\cap \overline{U^\sigma}=\varnothing$ for   a sufficiently small $\sigma>0$  and  all $\nu\in[0,1]$.

Let us extend the isotopy $\alpha_\nu$   to an isotopy $\overline\alpha_\nu:X\to X^\sigma= W\setminus U^\sigma$ which is fixed   on $(\Op\p W)\cap X$. Note that  the push-forward form $ (\overline\alpha_1)_*\wh\Lambda$  coincides with $\mu$ on $\Op\alpha_1(H_\intr)$. Hence, we can define  a Liouville form $\oLambda$ on $X^\sigma=W\setminus (U^\sigma\cup\Sigma)$ by setting it  equal to 
 $  (\overline\alpha_1)_*\wh\Lambda$ on   $\overline{\alpha}_1(X)$, and equal to $\mu$ elsewhere on  $X^\sigma\setminus\Sigma$.
It remains to observe that the form $\Lambda^\sigma:= \oLambda+\tau$ extends  smoothly to $X^\sigma$ and has the required properties. This concludes the proof of Proposition \ref{prop:main3}.
\end{proof}
 %%%%%%%%%%%%%%%%%%%% 

\end{document}